\documentclass[11pt]{amsart}

\usepackage[latin1]{inputenc}
\usepackage{amssymb,amsmath,amsfonts}
\usepackage{enumerate}
\usepackage{hyperref}
\usepackage{hyperref}
\usepackage[noadjust]{cite}

\usepackage[a4paper,twoside,left=3.2cm,right=3cm,top=3.1cm,bottom=2.3cm]{geometry}
\usepackage{mathtools, nccmath}

\newtheorem{proposition}{Proposition}[section]
\newtheorem{definition}[proposition]{Definition}
\newtheorem{lemma}[proposition]{Lemma}
\newtheorem{theorem}[proposition]{Theorem}
\newtheorem{corollary}[proposition]{Corollary}

\newcounter{theor}

\newtheorem{teor}[theor]{Theorem}

\theoremstyle{definition}
\newtheorem{remark}[proposition]{Remark}

\newcommand{\vol}{V_n}

\DeclareMathOperator{\kl}{Kl}
\DeclareMathOperator{\Val}{{\bf Val}}

\DeclareMathOperator{\Kl}{Kl}
\DeclareMathOperator{\Gr}{Gr}
\DeclareMathOperator{\spa}{span}
\DeclareMathOperator{\SL}{SL}
\DeclareMathOperator{\GL}{GL}

\DeclareMathOperator{\SO}{SO}
\renewcommand{\O}{\mathrm{O}}

\newcommand{\R}{\mathbb{R}}

\newcommand{\N}{\mathbb{N}}
\newcommand{\sfe}{\mathbb{S}^{n-1}}
\newcommand{\ball}{B^n}
\newcommand{\sfed}{\mathbb{S}^1}

\newcommand{\K}{\mathcal{K}}

\newcommand{\A}{\mathcal{A}}

\newcommand{\di}{\Phi} 
\newcommand{\dib}{\Psi} 
\newcommand{\rv}{\mu} 
\newcommand{\rvd}{\Phi} 
\newcommand{\gv}{\varphi} 

\newcommand{\x}{\overline{x}}
\renewcommand{\v}{\overline{v}}

\newcommand{\pv}{\overline{p(v)}}
\newcommand{\vo}{\overline{v_0}}

\DeclareMathOperator{\MVal}{{\bf MVal}}
\DeclareMathOperator{\MAdd}{{\bf MAdd}}
\DeclareMathOperator{\MEnd}{{\bf MEnd}}

\newcommand{\st}{\mathrm{st}}

\DeclareMathOperator{\V}{\mathrm{V}}

\newcommand{\func}[5]{\ensuremath{\begin{array}{cccl}
#1:&#2&\longrightarrow&#3\\&#4&\mapsto&#5\end{array}}}

\title[Minkowski additive operators under volume constraints]{Minkowski additive operators \\under volume constraints}
\author{Judit Abardia-Ev\'equoz} 
\address{Institut f\"ur Mathematik, Goethe-Universit\"at Frankfurt am Main, 
Robert-Mayer-Str. 10, 60054 Frankfurt, Germany}
\email{abardia@math.uni-frankfurt.de}

\author{Andrea Colesanti} 
\address{Dipartimento di Matematica ``U. Dini'', Viale Morgagni 67/A, 50134 Firenze, Italy}
\email{colesant@math.unifi.it}

\author{Eugenia Saor\'in G\'omez} 
\address{Institut f\"ur Algebra und Geometrie, Universit\"at Magdeburg, 
Universit\"atsplatz 2, 39106 Magdeburg, Germany}
\email{eugenia.saorin@ovgu.de}

\begin{document}

\thanks{The first author is supported by the DFG grant AB 584/1-2. The second author is supported by the FIR project 2013 
``Geometrical and Qualitative aspects of PDEs'', and by the GNAMPA.
The third author is supported by 19901/GERM/15, Fundaci\'on S\'eneca, CARM, Programa de Ayudas a Grupos de Excelencia de la Regi\'on de Murcia}

\date{\today}



\begin{abstract} 
We investigate Minkowski additive, continuous, and translation invariant operators $\di:\K^n\to\K^n$ defined on the family of convex bodies such that the volume of the image $\di(K)$ is bounded from above and below by multiples of the volume of the convex body $K$, uniformly in $K$. 
We obtain a representation result for an infinite subcone contained in the cone formed by this type of operators. Under the additional assumption of monotonicity or $\SO(n)$-equivariance, we obtain new characterization results for the difference body operator.
\end{abstract}

\maketitle

\section{Introduction}

Let $\K^n$ be the space of convex bodies in $\R^n$, i.e., compact convex subsets of $\R^n$, endowed with the usual Minkowski addition. 
An operator $\di:\K^n\longrightarrow\K^n$ is called {\em linear}, or {\em Minkowski additive}, if:
\begin{equation*}\label{linear}
\di(K+L)=\di(K)+\di(L),\quad\forall\, K,L\in\K^n.
\end{equation*}
In this paper we will consider:

\begin{itemize}\itemsep3pt
\item[{\footnotesize \bf{$\MAdd$}}:] Minkowski additive, continuous (with respect to the Hausdorff metric) and translation invariant operators. Their class will be denoted by $\MAdd$;
\item[{\footnotesize \bf{$\MEnd$}}:] Minkowski additive, continuous, and translation invariant operators, which are additionally {\em rotation equivariant}, i.e., 
they commute with rotations of $\R^n$. Following the current notation (see for instance \cite{schneider.book14} or
\cite{dorrek}), we will refer to these operators as {\em Minkowski endomorphisms} and denote their class by $\MEnd$.
\end{itemize}

Minkowski endomorphisms were first introduced in 1974 by Schneider \cite{schneider74}. In fact, in \cite{schneider74} the author considers operators 
$\di:\K^n\longrightarrow\K^n$,  which are continuous, linear, and commute with Euclidean motions. 
Subsequently, Kiderlen  \cite{kiderlen} proved that any operator of this form is, up to the addition of the Steiner point, a Minkowski endomorphism in the sense of the 
previous definition.

\medskip

A significative example of Minkowski endomorphism is given by the so-called \emph{difference body operator} 
\begin{equation}\label{DK}
\func{D}{\K^n}{\K^n}{K}{DK:=K+(-K),}
\end{equation} where + denotes the Minkowski addition and $-K=\{x\in\R^n\,:\,-x\in K\}$. We refer the reader to Section~\ref{preliminaries} for precise definitions and notation. 

\medskip

Since their introduction, Minkowski endomorphisms and their extensions have been widely studied and different representation results have been obtained, 
usually under further geometric assumptions (see, for instance, \cite{kiderlen,schneider74.dim2,schneider74.additive,schuster.07,schuster.10,schuster.wannerer.smooth,schuster.wannerer16}).  We also point out the recent work by Dorrek~\cite{dorrek}, who carried out a deep systematic study of Minkowski 
endomorphisms closing some of the most important conjectures related to them within convex geometry.

\medskip

In this paper we focus on Minkowski additive operators (in $\MAdd$ or in $\MEnd$) verifying a natural {\em affine isoperimetric inequality}, involving the volume. 
Affine isoperimetric inequalities -- i.e.\ those relating two geometric quantities associated to a convex body such that their ratio is invariant under the action of affine 
transformations -- have been largely studied in convex geometry. They turned out to have important applications, for instance in improving well-known inequalities 
in analysis. We refer the reader to the survey \cite{lutwak.handbook} and 
to \cite{cianchi_lyz_2009,haberl.schuster1,haberl.schuster2,haberl.schuster.xiao,lyz_2000,zhang,wang,haddad.jimenez.montenegro,lutwak86} and references therein for some 
recents results and applications.

Our starting point consists of two inequalities of this sort, associated to the difference body:
\begin{equation}\label{RS}
2^n\vol(K)\leq\vol(DK)\leq\binom{2n}{n}\vol(K),\quad\forall K\in\K^n.
\end{equation}
The right-hand inequality is the celebrated \emph{Rogers-Shephard}, or {\em difference body}, inequality proved in \cite{rogers.shephard} 
(see also \cite{chakerian,rogers.shephard58} for other proofs and related inequalities). The left one can be directly obtained from 
the Brunn-Minkowski inequality (see \cite{schneider.book14}). Note that both inequalities are known to be sharp, and equality cases are completely characterized.

Motivated by \eqref{RS}, we introduce the following definition.

\begin{definition}\label{def VC}
Let $\di:\K^n\longrightarrow\K^n$. We say that $\di$ \emph{satisfies a  volume constraint (VC)} if there are constants $c_{\di},C_{\di}>0$ such that
\begin{equation}\tag{VC}
c_{\di}\vol(K)\leq \vol(\di(K))\leq C_{\di}\vol(K),\quad\forall K\in\K^n.
\end{equation}
\end{definition}

Our first result is a characterization of Minkowski endomorphisms verifying a (VC) condition. Roughly speaking, we prove that any such
application is a non-symmetric version of the difference body operator. If in addition we assume that the images of the operator are 
symmetric with respect to the origin (we denote by $\K^n_s$ the class of convex bodies with this property), then the difference body operator is characterized. 

\begin{theorem}\label{+On_dim_geq_3}
Let $n\geq 3$.
\begin{enumerate}
\item[\emph{(i)}] An operator $\di:\K^n\longrightarrow\K^n$ is a Minkowski endomorphism that satisfies (VC) if and only if there are $a,b\geq 0$ with $a+b>0$ such that 
$$
\di K=a(K-\st(K))+b(-K+\st(K)),\quad\forall K\in\K^n.
$$
\item[\emph{(ii)}] An operator $\di:\K^n\longrightarrow\K^n_s$ is a Minkowski endomorphism that satisfies (VC) if and only if there is a $\lambda>0$ such that $\di K=\lambda DK$ for every $K\in\K^n$.
\end{enumerate}
\end{theorem}
Here 
$$
\st(K)=\frac{1}{\vol(B^n)}\int_{\sfe} u h(K,u)du
$$ 
denotes the Steiner point of $K$, $B^n$ and $\sfe$ are the unit ball and the unit sphere in $\R^n$, respectively, and $h(K,\cdot)$ is the support function of $K$.
The case $n=2$ is considered in Theorem~\ref{+On_dim_2}, where the result is analogous up to a rotation about the origin.

\smallskip
Another important geometrical property in convex geometry is the monotonicity. An operator $\di:\K^n\longrightarrow\K^n$ is \emph{monotonic} if for every $K,L\in\K^n$ with $K\subseteq L$, then $\di(K)\subseteq\di(L)$. In the context of Minkowski endomorphisms, the more general notion of weak monotonicity turns out to be more relevant. We say that an operator $\di:\K^n\longrightarrow\K^n$ is \emph{weakly monotonic} if for every $K,L$ with $\st(K)=\st(L)=0$ and $K\subseteq L$, the inclusion $\di(K)\subseteq\di(L)$ holds. 
Weakly monotonic endomorphisms were characterized by Kiderlen \cite{kiderlen}, who proved that they are given by certain measures on the sphere. 
Before the recent results by Dorrek \cite{dorrek} came to light, all known Minkowski endomorphism were weakly monotonic. 
On the other hand, Theorem~\ref{+On_dim_geq_3} proves, in particular, that all Minkowski endomorphisms satisfying (VC) are weakly monotonic. 

\smallskip

These two results lead us to our second result which provides a classification of Minkowski additive operators in $\MAdd$ that are monotonic.

\begin{theorem}\label{+mon intro}
Let $n\geq 2$ and $\di\in\MAdd$. Then $\di$ is monotonic and satisfies (VC) if and only if there exists $g\in\GL(n)$ such that 
$$
\di(K)=g(DK),\quad\forall\, K\in\K^n.
$$
\end{theorem}

\medskip

Nowadays, Minkowski additive operators are usually treated as a special case of {\em Minkowski valuations}. 
A Minkowski valuation is an operator $\di:\K^n\longrightarrow\K^n$ such that 
$$\di(K)+\di(L)=\di(K\cup L)+\di(K\cap L),\quad\forall K,L\in\K^n,\,K\cup L\in\K^n,$$
where $+$ denotes Minkowski addition in both sides of the equality.  

The relation between Minkowski additive operators and Minkowski valuations was first studied (under different names) by Spiegel~\cite{spiegel}, who proved 
that every continuous, translation invariant, and $1$-homogeneous Minkowski valuation is Minkowski additive. 
We recall that an operator $\di: \K^n\longrightarrow \K^n$ is said to be $1$-homogeneous if $\di(\lambda K)=\lambda K$ for all $\lambda> 0$.
From the relation 
\begin{equation}\label{id is val}
K+L=K\cup L+K\cap L,
\end{equation} 
which holds for every $K,L,K\cup L\in\K^n$ (see, e.g., \cite[Lemma 3.1.1]{schneider.book14}), it follows directly that every Minkowski additive 
operator is a Minkowski valuation.  
In this sense, Schneider's study of Minkowski endomorphisms in \cite{schneider74,schneider74.dim2} can be considered as the starting point of the investigation of  Minkowski valuations. In~\cite{schneider74}, various classification results for Minkowski endomorphisms are obtained by adding different geometric properties, such as preserving some precise intrinsic volume, or prescribing the image of a segment (see Theorem~\ref{SO_schneider}). Another remarkable result by Schneider~\cite{schneider74.additive} classifies the operators between convex bodies that preserve the volume. This result can be interpreted as the first one considering the (VC) condition in  the special case that both constants are equal to one. 
\begin{teor}[\cite{schneider74.additive}]
Let $n\geq 2$. An operator $\di:\K^n\longrightarrow\K^n$ is Minkowski additive and satisfies $\vol(\di(K)) = \vol(K)$ for every $K\in\K^n$ if and only if there exist $\alpha\in\SL(n)$ and a Minkowski additive map $t:\K^n\longrightarrow\R^n$ such that $\di(K)=\alpha K+t(K)$ for every $K\in\K^n$.
\end{teor}

A systematic study of Minkowski valuations was started by Ludwig~\cite{ludwig02,ludwig05} at the beginning of this century. We point out her characterization of the difference body operator.
\begin{teor}[\cite{ludwig05}]\label{t: ludwig class DK} Let $n\geq 2$. An operator $\di: \K^n\longrightarrow\K^n$ is a continuous, translation invariant, and 
$\SL(n)$-covariant Minkowski valuation if and only if there is a $\lambda\geq 0$ such that $\di(K)=\lambda DK$ for every $K\in\K^n$.
\end{teor}
We recall that  $\di: \K^n\longrightarrow\K^n$ is said to be {\em covariant} with respect to a group $G$ of transformations of 
$\R^n$ if $\di(g(K))=g(\di(K))$ for every $K\in\K^n$ and $g\in G$.

After these seminal papers of Ludwig, Minkowski valuations have been deeply studied and characterization results for other operators, for other groups of transformations and for 
certain subfamilies of $\K^n$ have been obtained. Some of these results can be found in 
\cite{kiderlen,schuster.wannerer.smooth,schuster.wannerer16,haberl,schuster.10,schuster.wannerer12,wannerer.equiv,schneider_schuster06,parapatits_wannerer} 
and references therein. 

Apart from the mentioned characterization results of the difference body operator, by Schneider and Ludwig, more recent results in this direction have been proven. For instance, 
Gardner, Hug, and Weil~\cite{gardner.hug.weil1}, managed to remove the Minkowski valuation property in Theorem~\ref{t: ludwig class DK}, although adding the condition of homogeneity. This 
result was obtained in the much more general context of classifying operations between convex bodies (see also \cite{bianchi.gardner.gronchi,gardner.hug.weil2, milman.rotem}). 
Other characterization results for the difference body were established in \cite{abardia.saorin2}, by using the notion of satisfying a (VC) condition and the $\GL(n)$-covariance.

\smallskip
We denote by $\MAdd^{s,+}$ the subset of those operators of $\MAdd$ that verify $\di(-K)=\di(K)=-\di(K)$ for every $K\in\K^n$. 
The above characterization results, Theorem~\ref{+On_dim_geq_3} and~\ref{+mon intro}, are obtained as consequence of a representation result for 
operators $\di\in\MAdd^{s,+}$ that satisfy (VC).
The corresponding result in dimension higher than $2$ will be stated in Section~\ref{s: representation formulae}, as it requires further notation
and additional definitions.

\begin{theorem}\label{thm n=2} 
Let $n=2$. An operator $\di\in\MAdd^{s,+}$ satisfies (VC) if and only if there exist
$\rho\,:\,\Gr(2,1)\longrightarrow\R$, continuous and strictly positive, and $\pi\,:\,\Gr(2,1)\longrightarrow\Gr(2,1)$, bijective and bi-Lipschitz, such that for every $K\in\K^2$,
$$h(\di(K),u)=\int_{\sfed}
\rho(\v) V_1([-u,u]|\pi(\v))
dS_1(K,v),\quad\forall\, u\in\sfed.
$$
Here $S_1(K,\cdot)$ denotes the surface area measure of $K\in\K^n$, $V_1(L)$ denotes the length of the 1-dimensional convex body $L$, and $\overline{v}$ denotes the vector obtained from a $\pi/2$-counterclockwise rotation of $v$.
\end{theorem}
The proof of this representation result, as well as its higher dimensional analogue relies
on the existence of  bi-Lipschitz bijections on $\Gr(n,k)$ for every $1\leq k\leq n-1$, associated to every $\di\in\MAdd$ that satisfies (VC).

\medskip
The paper is organized as follows. 
In Section~\ref{preliminaries}, we recall the basic concepts in convex geometry and theory of valuations that will be used throughout the paper. 
In Section~\ref{section 1-homogeneous} we investigate properties of operators from $\MAdd$, focusing in particular on their interaction 
with Grassmannians and proving that they preserve dimensions. 
In Section~\ref{s: representation formulae} we obtain representation formulae for Minkowski additive operators that satisfy the (VC) 
condition and prove Theorem~\ref{thm n=2}. 
Theorem~\ref{+mon intro} is proved in Section~\ref{sec: 1h and mon}. 
The last section is devoted to the characterization contained in Theorem~\ref{+On_dim_geq_3}, of Minkowski endomorphisms that satisfy a (VC) condition.

\section{Preliminaries}\label{preliminaries}

\subsection{Notation}
We work in the $n$-dimensional Euclidean space $\R^n$, equipped with usual scalar product $\langle\cdot,\cdot\rangle$ and norm $\|\cdot\|$. If $A\subset\R^n$ is a measurable set, $\vol(A)$ denotes its volume, that is, its $n$-dimensional Lebesgue measure. The notation
$\sfe$ and $\ball$ stands for the unit sphere and the unit ball (centered at the origin) of $\R^n$, respectively. The standard basis of $\R^n$ is 
denoted by $\{e_1,\dots,e_n\}$. For $x\in\R^n$, we denote by $S_x:=[-x,x]$ the line segment joining $-x$ and $x$. 
For $x=(x_1,x_2)\in\R^2\setminus\{0\}$, we denote by $\x\in\R^2$, the orthogonal vector to $x$ given by $(-x_2,x_1)$.

As usual, $\Gr(k,n)$ denotes the Grassmannian of linear $k$-dimensional subspaces of $\R^n$.
For $A\subset \R^n$, we denote by $A|E$ the {\em orthogonal projection} of the set $A$ onto $E\in\Gr(k,n)$.
If $E$ is a linear subspace of $\R^n$, we denote by $E^{\perp}$ the orthogonal complement of $E$ in $\R^n$.
We write $\GL(n)$ and $\SL(n)$ to denote the general linear and special linear groups in $\R^n$. By $\O(n)$ we denote 
the group of orthogonal transformations of $\R^n$ and by $\SO(n)\subset\O(n)$ the orthogonal transformations which preserve orientation.

If $A\subset\R^n$, then $\spa A$, the span of $A$, is the vector subspace of $\R^n$ parallel to the affine subspace of smallest dimension containing $A$. 
The \emph{dimension} of a set $A\subset\R^n$, $\dim A$, is defined as the dimension of $\spa A$.

\subsection{Convex bodies}
Next we recall some notions from convex geometry, which will be used throughout the paper. Our reference text for this part is the monograph \cite{schneider.book14}
by Schneider. We refer the reader also to the books \cite{gardner.book06,gruber.book,artstein.giannopoulos.milman} for different perspectives on this subject.

We denote by $(\K^n,+)$ the set of convex bodies (compact and convex sets) in $\R^n$, endowed with the usual Minkowski addition:
$$
K+L:=\{x+y\,:\, x\in K,\, y\in L\}.
$$
The topology that we consider on $\K^n$ is the one induced by the Hausdorff distance.
 For $K,L\in\K^n$, the \emph{Hausdorff distance} between $K$ and $L$ is given by 
 $$d_H(K,L)=\max\{\min\{\epsilon>0\,:\,K\subset\epsilon L\},\min\{\epsilon>0\,:\,L\subset\epsilon K\}\}.$$

By $\K^n_s$, we denote the set of convex bodies in $\R^n$ which are symmetric with respect to the origin. We call \emph{$o$-symmetric} bodies the elements of $\K^n_s$.

The \emph{support function} $h(K,\cdot)\,:\,\R^n\longrightarrow\R$ of a convex body $K\in\K^n$ is given by
$$
h(K, v)=\max \{ \langle v, x\rangle : x \in K \}.
$$
This is a 1-homogeneous convex function in $\R^n$ which determines uniquely $K$ (\cite[Theorem 1.7.1]{schneider.book14}). Moreover, for any $u\in\R^n$, the function $h(\cdot,u):\K^n\longrightarrow \R$ is additive with respect to the Minkowski addition and
positively homogeneous, namely,
\begin{equation}\label{additivity of h in K}
h(\alpha K+\beta L,\cdot)=\alpha h(K,\cdot)+\beta h(L,\cdot),\quad\forall\, K,L\in\K^n,\, \forall\, \alpha,\beta\ge0.
\end{equation}
Further, for any $K\in\K^n$, $g\in \GL(n)$, and $u\in\R^n$,
\begin{equation}\label{supp func g^t}
h(gK,u)=h(K,g^Tu),
\end{equation}
where $g^T$ denotes the transpose of the matrix $g\in\GL(n)$. 

We note that Hausdorff topology on $\K^n$ is equivalent to the uniform convergence topology in the set of support functions (see, e.g., \cite[Lemma 1.8.14]{schneider.book14}).

\medskip
A {\em zonoid} is a convex body which can be approximated by finite sums of line segments.
A convex body $Z\in\K^n$ is a {\em generalized zonoid} if its support function can be written in the form
\begin{equation}\label{gen_zonoid}
h(Z,u)=\int_{\sfe}|\langle u,v\rangle|d\rho_Z(v)+\langle u,v_0\rangle,\quad\forall u\in\R^n,
\end{equation}
where $\rho_Z$ is a signed even measure on $\sfe$, called the \emph{generating measure of $Z$}, and $v_0\in\R^n$ is fixed. In particular every generalized zonoid has a center of symmetry (which is the origin when $v_0=0$). 
It is known that generalized zonoids are dense in $\K^n_s$ (see \cite[Corollary 3.5.7]{schneider.book14}).

\subsection{Mixed volumes}\label{s: mixed vol}
The mixed volume in $\R^n$ is the unique multilinear functional 
\[
\func{V}{(\K^n)^n }{\R}{(K_1,\dots,K_n)}{V(K_1,\dots,K_n)}
\]
which satisfies $V(K,\dots,K)=\vol(K)$ and is symmetric and Minkowski additive in each component. The mixed volume functional is non-negative and in each variable it is continuous, translation invariant, and monotonic. We refer the reader to~\cite[Chapter~5]{schneider.book14} for a systematic study of mixed volumes.

For $1\leq i\leq n$ and $K,K_{i+1},\dots,K_n\in\K^n$, we use $[i]$ inside a mixed volume to denote that the convex body $K$ is repeated $i$ times:
\[
V(K[i],K_{i+1},\dots,K_n)=V(\underbrace{K,\dots,K}_{i\textrm{-times}},K_{i+1},\dots, K_n).
 \] 
Some special cases of mixed volumes give rise to well-known quantities such as the surface area, the Euler characteristic, or more generally the \emph{intrinsic volumes} that are defined as
$$V_j(K)=\frac{1}{V_{n-j}(B^{n-j})}\binom{n}{j}V(K[j],B^n[n-j]),\quad K\in\K^n.$$
If $j=n$, then $V_n(K)$ coincides with the $n$-dimensional volume of $K$, $V_{n-1}(K)$ with its surface area and $V_0(K)$ with the Euler characteristic of $K$.

Let $E\in\Gr(n,k)$, let $K_1,\dots,K_k\in\K^n$, and let $L_1,\dots,L_{n-k}$ be  convex bodies contained in $E^{\perp}\in\Gr(n,n-k)$. 
The mixed volume $V(K_1,\dots,K_k,L_1,\dots,L_{n-k})$ can be
split in a product of mixed volumes of convex bodies contained in $E$ and $E^{\perp}$, as follows:
\begin{equation}\label{eq_Gardner}
\binom{n}{k}V(K_1,\dots,K_k,L_1,\dots,L_{n-k})=V_E(K_1|E,\dots,K_k|E)V_{E^{\perp}}(L_1,\dots,L_{n-k}),
\end{equation}
where $V_E$ and $V_{E^{\perp}}$ denote the mixed volume functional defined on $E$ and $E^{\perp}$ respectively, with
the standard identification with $\R^k$ and $\R^{n-k}$, respectively (see, e.g., \cite[(A.36)]{gardner.book06}).

An immediate consequence of~\eqref{eq_Gardner} is the following representation of the support function of the difference body.

\begin{corollary}\label{th repres D_v} Let $n\geq 2$, $v\in\R^n$ and let $K\in\K^n$. 
Then there exist $n-1$ vectors $\{v_2^v,\dots,v_n^v\}$, constituting an orthonormal basis of  $\spa\{v\}^{\perp}$,  such that 
$$h(DK,v)=h(K,v)+h(K,-v)=\frac{1}{2^{n-1}}V(K,S_{v_2}^v,\dots,S_{v_n}^v),\quad\forall v\in\sfe,$$
for $S_{v_j}^v=[-v_j^v,v_j^v]$, $j\in\{2,\dots,n\}$. 
\end{corollary}
\begin{proof}
Observing that $h(K,v)+h(K,-v)$, the width of $K$ in the direction $v$, measures the length of the projection of $K$ onto the line spanned by $v$, we have $V_1(K|\spa\{v\})=h(K,v)+h(K,-v)$. Taking $E=\spa\{v\}$ in \eqref{eq_Gardner}, we obtain the result.  
\end{proof}

The following result provides equivalent conditions ensuring that a mixed volume is strictly positive.
\begin{theorem}[Theorem 5.1.8 in \cite{schneider.book14}]\label{mix_volumes_Schneider}
For $K_1,\dots,K_n\in\K^n$, the following assertions are equivalent:
\begin{enumerate}
\item[\emph{(a)}] $V(K_1,\dots,K_n)>0$;
\item[\emph{(b)}] there are segments $S_i\subset K_i$ $(i=1,\dots,n)$ having linearly independent directions;
\item[\emph{(c)}] $\dim(K_{i_1}+\dots+K_{i_k})\geq k$ for each choice of indices $1\leq i_1<\dots<i_k\leq n$  and for all $k\in\{1,\dots,n\}$.
\end{enumerate}
\end{theorem}

The \emph{surface area measure of order 1} of $K\in\K^2$, denoted by $S_1(K,\cdot)$, is the unique finite Borel measure on $\sfed$ such that
$$\V(K,B^2)=\dfrac{1}{2}\int_{\sfed}\,dS_1(K,u).$$
We refer the reader to \cite[Chapters 4 and 5]{schneider.book14} for a description of the surface area measure of order 1, and more generally, for the description of the surface area measures. 
Some properties of the surface area measure of order 1 that we need in the following are that $S_1(K,\cdot)$ is weakly continuous with respect to $K$ and it is Minkowski additive, namely, for $K,L\in\K^2$ and a Borel set $\omega\subset\sfed$, $S_1(K+L,\omega)=S_1(K,\omega)+S_1(L,\omega)$.
Further, for every Borel set $\omega\subset \sfed$, the map $K\mapsto S_1(K,\omega)$ is a real-valued, 1-homogeneous, and translation invariant valuation (see Equation \eqref{valuation property}).

\subsection{Translation-invariant valuations}\label{ti val}
Let $(\mathcal{A},+)$ be an Abelian semigroup. An operator $\gv:\K^n\longrightarrow\mathcal{A}$ is a \emph{valuation} if for any $K,L\in \K^n$ with $K\cup L \in \K^n$,
\begin{equation}\label{valuation property}
\gv(K)+\gv(L)=\gv(K\cup L)+\gv(K\cap L).
\end{equation}
The most well-known valuations are \emph{real-valued valuations}, i.e., those for which $(\mathcal A,+)=(\R,+)$ with the usual addition of real numbers. Real-valued valuations were probably first used by Dehn for his solution of the third Hilbert problem. The reader interested in the state of the art of the theory of real-valued valuations is referred to the valuable surveys \cite{hadwiger,alesker.survey,fu.survey,bernig.survey, klain.rota, mcmullen93, mcmullen.schneider83} 
and \cite[Chapter 6]{schneider.book14}, and to  \cite{alesker.bernig.schuster,alesker.faifman,bernig.fu.solanes,bernig.fu,ludwig.reitzner} for the most recent results. 
Nowadays, apart from real-valued and Minkowski valuations, other valuations, namely, for other abelian semigroups $(\mathcal A,+)$ than the reals or $(\K^n,+)$, have been studied, often motivated by their applications in material science and physics. Among these valuations are the tensor-valued valuations, area and curvature measures, or taking values in some space of  functions (see, for instance, \cite{bernig.hug,haberl.parapatits,ludwig_sob,ludwig.survey,ludwig.survey2,wannerer1,wannerer2}).

In this work, we will only consider {\em real-valued} and {\em Minkowski valuations}. Two basic features of valuations which will be used throughout are continuity with respect to the Hausdorff metric and translation invariance.
If $\gv:\K^n\longrightarrow \K^n$ is assumed to be continuous with respect to the Hausdorff distance, then the topology inherited from this distance is assumed in both, domain and image spaces.
A valuation $\gv$ (real-valued or Minkowski) is \emph{translation invariant} if 
\[
\gv(K+t)=\gv(K) \text{ for any } t\in\R^n\text{ and }K\in\K^n.
\]
A (real-valued or Minkowski) valuation $\gv$ is said to be \emph{homogeneous of degree} $k\in\R$ if for every $K\in\K^n$ and $\lambda>0$,
\[
\gv(\lambda K)=\lambda^k \gv(K).
\]
We say that a valuation $\gv:\K^n\longrightarrow (\mathcal A,+)$ taking values in an ordered semigroup $\A$ is \emph{monotonic} (increasing with respect to set inclusion) if for every $K,L\in\K^n$ such 
that $K\subseteq L$, we have $\gv(K)\subseteq\gv(L)$.

We denote by $\Val$ the family of continuous and translation invariant real-valued valuations. 
The subset of $\Val$ consisting of homogeneous of degree $k$ (resp.~even) real-valued valuations is denoted by $\Val_k$ (resp.~$\Val^+$). Analogously,
$\MVal$ and $\MVal_k$ are the spaces of continuous and translation invariant Minkowski valuations, and 
continuous, translation invariant, and $k$-homogeneous Minkowski valuations. Notice that all these sets of valuations have the structure of real vector spaces. 

An operator $\di:\K^n\longrightarrow\K^n$ is said to be an \emph{$o$-symmetrization} if $\di(K)$ is $o$-symmetric for every $K\in\K^n$. 

\smallskip
Real-valued and Minkowski valuations are naturally connected by means of the support function through the following construction. Let $\di$ be a Minkowski valuation and let 
$w\in\R^n$ be fixed. Then $\rvd_w\,:\,\K^n\longrightarrow\R$, defined by
\begin{equation}\label{white star}
\rvd_w(K)=h(\di(K),w),\quad\forall\, K\in\K^n,
\end{equation}
is a real-valued valuation which inherits several of the properties of $\di$.
In the next lemma we collect some properties of $\di_w$ that we will use all over the work without  explicit mention. 
\begin{lemma}\label{osservazione 1} Let $\di\in\MVal$, let $w\in\R^n$ and define
$$\func{\rvd_w}{\K^n}{\R}{K}{h(\di(K),w).}$$
The following facts hold:
\begin{enumerate}
\item[\emph{(i)}] $\rvd_w\in\Val$;
\item[\emph{(ii)}] if $\di$ is $k$-homogeneous, then $\rvd_w$ is $k$-homogeneous;
\item[\emph{(iii)}] if $\di$ is monotonic increasing (resp. decreasing), then $\rvd_w$ is monotonic increasing (resp. decreasing), too.
\end{enumerate}
Moreover, if we assume that $\di:\K^n\longrightarrow\K^n$ is a translation invariant and monotonic Minkowski valuation, then 
\begin{enumerate}
\item[\emph{(iv)}] 
$\di_w$ and $\di$ are continuous.
\end{enumerate} 
\end{lemma}

\begin{proof} 
 Let $\di\in\MVal$ and $w\in\R^n$.
(i) follows immediately from \eqref{id is val} and \eqref{additivity of h in K}. The proof of (ii) is a consequence of \eqref{additivity of h in K} and the homogeneity of $\di$. Finally, (iii) follows from the fact that $K\subseteq L$ if and only if $h(K,u)\leq h(L,u)$ for all $u\in\R^n$, together with the monotonicity of $\di$.
Property (iv) for $\di_w$ follows by a result of McMullen \cite{mcmullen77} stating that every translation invariant and monotonic real-valued valuation is continuous. Finally, the continuity of $\di$ is obtained from the continuity of $\di_w$ and the fact that the continuity with respect to the Hausdorff topology is equivalent to the uniformly convergence topology in the set of support functions. 
\end{proof}

\smallskip
Next we introduce the {\em Klain function} (\cite[p.~356]{schneider.book14}) of a real-valued valuation, which will be needed to prove Proposition~\ref{characterization}.

Let $\rv\in\Val_j$. The \emph{Klain function} of $\rv$ is the continuous map $\Kl_{\mu}:\Gr(n,j)\longrightarrow\R$ such that for every
$K\in\K^n$ such that $K\subset E\in\Gr(n,j)$,
$$
\rv(K)=\Kl_\rv(E)\, V_j(K)
$$ 
where $V_j$ denotes the $j$-th intrinsic volume of $K\subset E$, that is, its $j$-dimensional volume.
For $\mu\in\Val_1$ and $E\in\Gr(n,1)$, the definition of $\Kl_\rv(E)$ yields that
\begin{equation}\label{identity 2}
\Kl_\rv(E)=\frac12\rv(E\cap\ball).
\end{equation}
Klain proved in \cite{klain00} that homogeneous and even real-valued valuations are uniquely determined by its Klain function (see \cite[Theorem 6.4.11]{schneider.book14}). More precisely:
\begin{theorem}[\cite{klain00}]\label{klain_injectivity}
The map $\Val_j^+\longrightarrow C(\Gr(n,j))$ is injective. 
\end{theorem}

Next we state another result of Klain, for the particular case that we need in the following.
\begin{theorem}[Theorem 6.4.12 in \cite{schneider.book14}]\label{6.4.12}
Let $\rv\in\Val_1$ and let $\kl_{\rv}:\Gr(n,1)\longrightarrow\R$ be its Klain function. If $Z\in\K^n$ is a generalized zonoid with generating measure $\rho_Z$, then 
$$\rv(Z)=2\int_{\sfe}\kl_{\rv}(E_u)d\rho_Z(u),$$
where $E_u=\spa\{u\}$.
\end{theorem}

\subsection{The space $\MAdd$}
We consider an operator $\di:\K^n\longrightarrow\K^n$ which it is continuous (with respect to the Hausdorff topology), translation invariant, 
and \emph{Minkowski additive}, that is, 
$$
\di(K+L)=\di(K)+\di(L),\quad\forall K,L\in\K^n.
$$
We denote the space of these operators by $\MAdd$. The subspace of even (resp.~$o$-symmetrizations) is denoted by $\MAdd^{+}$ (resp. $\MAdd^s$). 

As described in the introduction, we have the following relation. 
\begin{lemma}[\cite{spiegel}, Remark 6.3.3 in \cite{schneider.book14}]
Let $n\geq 2$. Then $\di\in\MVal_1$ if and only if $\di\in\MAdd$. 
\end{lemma}

The following lemma describes the image of a point of an operator $\di\in\MAdd$.
\begin{lemma}\label{additive}
Let $n\geq 2$. If $\di\in\MAdd$, then $\di(\{p\})=0$  for every $p\in\R^n$.
\end{lemma}

\begin{proof}
Since $\di$ is Minkowski additive and translation invariant, we have that for every convex body $K$, 
\begin{align*}
\di(K+p)&=\di(K)+\di(\{p\})\\&=\di(K),
\end{align*}
which implies $\di(\{p\})=0$ for every $p\in\R^n$.
\end{proof}

In Section~\ref{sec: 1h and mon} we will study operators belonging to $\MAdd$, satisfying (VC), which are also monotonic.  We will need the following results in this direction.

\begin{theorem}[\cite{firey}, \cite{milman.schneider}]\label{Firey_mon}
Let $\rv:\K^n\longrightarrow\R$ be a monotonic, translation invariant, and 1-homogeneous valuation. Then there exist $k\in\{1,\dots,n\}$ and $(n-k)$ 
pairwise orthogonal unit segments $S_{k+1},\dots,S_n$ such that one of the following cases occurs. 
\begin{itemize}
\item[(i)] $k\geq 2$ and there is a convex body $L$ of dimension $k$, contained in the orthogonal complement of $\mathrm{span}\{S_{k+1},\dots,S_n\}$ 
such that 
$$
\rv(K)=V(K,L[k-1],S_{k+1},\dots,S_n).
$$
\item[(ii)] $k=1$ and there is a constant $c>0$ such that 
$$
\rv(K)=cV(K,S_{2},\dots,S_n).
$$
\end{itemize}
\end{theorem}

Notice that the segments $S_{k+1},\dots,S_n$ can be chosen to be centered at the origin since mixed volumes are translation invariant in each component.

We would like to remark that Corollary~\ref{th repres D_v} also follows from the above representation result. In Section 7, we include a proof which enlightens also the proof of Theorem~\ref{+mon intro}. 

In Section~\ref{sec: SOn} we will investigate \emph{Minkowski endomorphisms}, that is, operators belonging to $\MAdd$ that are $\SO(n)$-equivariant. 
We say that $\di:\K^n\longrightarrow\K^n$ is \emph{$\SO(n)$-equivariant} if $$\di(gK)=g\di(K),\quad\forall g\in\SO(n),\,\forall K\in\K^n.$$
The space of Minkowski endomorphisms is denoted by $\MEnd$. 
For the study of Minkowski endomorphisms satisfying (VC), we will need the following results.

\begin{theorem}[\cite{alesker.bernig.schuster}]\label{SOimpliesO}
If $\di\in\MVal$ is $\SO(n)$-equivariant, then $\di$ is also $\O(n)$-equivariant.
\end{theorem}

\begin{theorem}[\cite{schneider74}]\label{SO_schneider} 
Let  $\di\in\MEnd$.
\begin{enumerate}
\item[$\emph{(i)}$]  Let $n=2$. Then the image under $\di$ of some convex body is a non-degenerated segment if and only if there are 
$a,b\geq 0$ with $a+b>0$ and $g\in\SO(2)$ such that 
$\di K=ag(K-\st(K))+bg(-K+\st(K))$ for every $K\in\K^2$.
\item[$\emph{(ii)}$] Let $n\geq 3$. Then the image under $\di$ of some convex body is a non-degenerated segment if and only if there are $a,b\geq 0$ with $a+b>0$ such that 
$\di K=a(K-\st(K))+b(-K+\st(K))$ for every $K\in\K^n$.
\end{enumerate}
\end{theorem}

\subsection{Grassmannians}\label{grass}
For $n\ge2$ and $k\in\{1,\dots,n\}$, we denote by $\Gr(n,k)$ the Grassmannian formed by all $k$-dimensional vector subspaces of 
$\R^n$. More generally, if $E\in\Gr(n,k)$ and $j\in\{1,\dots,k\}$, $\Gr(E,j)$ denotes the set of all $j$-dimensional subspaces of $E$.

The following map defines a distance on $\Gr(n,k)$
\begin{equation}\label{d_Gr}
\func{d}{\Gr(n,k)\times\Gr(n,k)}{[0,\infty)}{(E,F)}{d(E,F)=d_H(E\cap\sfe,F\cap\sfe),}
\end{equation}
 where $d_H$ is the Hausdorff distance. 

We observe that the distance on $\Gr(n,1)$ is equivalent to the angle distance on $\sfe$. Indeed, let $v_1,v_2\in\sfe$. If $E_1=\spa\{v_1\}$ and $E_2=\spa\{v_2\}$, then
\begin{equation}\label{dist_Grn1}
d(E_1,E_2)=2\sin\left(\frac{\alpha}{2}\right),
\end{equation} 
where $\alpha\in[0,\pi/2]$ such that $\cos(\alpha)=\langle v_1,v_2\rangle$. 
This distance can be directly computed by elementary trigonometry and recalling definition \eqref{d_Gr}. 

The above distance between $F_1,F_2$ in $\Gr(n,k)$ can be also expressed in terms of the angles between $F_1,F_2$. Indeed, by the above definition of distance on $\Gr(n,k)$ and the definition of the Hausdorff distance,  (see (1.59) in \cite{schneider.book14} and subsequent comments)
\begin{equation}\label{dist_angles_Grnk}
d(F_1,F_2)=\max\left\{
\max_{L_1\in\Gr(F_1,1)}\!\!\left(\min_{L_2\in\Gr(F_2,1)}d(L_1,L_2)\right)\!,
\max_{L_2\in\Gr(F_2,1)}\!\!\left(\min_{L_1\in\Gr(F_1,1)}d(L_1,L_2)\right)
\right\}.
\end{equation}
Hence, \eqref{dist_Grn1} yields the claim. 

\smallskip
For the proof of our main results, we will strongly use that the space $\Gr(n,k)$, $1\leq k\leq n-1$, is compact and connected, with the topology from the endowed distance. Since the authors were not able to find appropriate references of these two facts, a proof is sketched in the appendix.

\subsection{Volume constraints}
The main new geometric property that we consider in this paper is the following volume constraint for an operator $\di\in\MAdd$. 
\begin{definition} We say that an operator $\di:\K^n\longrightarrow\K^n$ satisfies a  volume constraint condition (VC) if there exist constants $c_{\di},C_{\di}>0$ such that
$$c_{\di}\vol(K)\leq \vol(\di(K))\leq C_{\di}\vol(K),\quad\forall K\in\K^n.$$
\end{definition}

The identity operator on $\K^n$ trivially satisfies (VC). Another example of an operator satisfying a volume constraint (VC) is the difference body operator \eqref{DK}, since it satisfies \eqref{RS}. However, there are much more other elements of $\MAdd$ satisfying (VC). They will be described in Section~\ref{s: representation formulae}.

\section{The (VC) condition interacts with Grassmannians}\label{section 1-homogeneous}

In this section, we investigate properties of operators of $\MAdd$ which satisfy a volume constraint condition (VC) and prove that associated to every $\di\in\MAdd^{s,+}$ which satisfies (VC) there is a bi-Lipschitz bijection on $\Gr(n,k)$ for every $1\leq k\leq n-1$, in Theorem~\ref{teo Grassmannian}.

In the next lemma we prove that an operator of $\MAdd$ which satisfies (VC) preserves dimensions.

\begin{proposition}\label{general_facts}
Let $n\geq 2$ and let $\di\in\MAdd$ satisfy (VC). Then
$$\dim(K)=\dim(\di K),\quad\forall K\in\K^n.$$
\end{proposition}

\begin{proof}
Let $n\geq 2$ and let $\di\in\MAdd$ satisfy (VC) with $c_{\di},C_{\di}>0$. We first observe that from Lemma~\ref{additive} we already have that if $x_0\in\R^n$ is a point,
then $0=\dim\{x_0\}=\dim \di(\{x_0\})=\dim\{0\}=0$.

We prove that $\di$ preserves dimensions for every $K\in\K^n$. Let $l\in\{1,\dots,n\}$, let $K\in\K^n$ with $\dim K= l$, and let $L\in\K^n$ with $\dim L=n-l$ and such that $\dim K+L=n$. For every $\lambda>0$, the Minkowski additivity of $\di$ and the (VC) condition yield
\begin{equation}\label{ineq_from_VC}
c\vol(K+\lambda L)\leq \vol(\di(K+\lambda L))=\vol(\di K+\lambda \di L)\leq C\vol(K+\lambda L).
\end{equation}
Using the expansion of $\vol(K+\lambda S)$ and $\vol(\di(K)+\di (\lambda S))$ as polynomials whose coefficients are mixed volumes (see \cite[Section 4.1]{schneider.book14}), together with Theorem~\ref{mix_volumes_Schneider},  we obtain 
$$V_n(K+\lambda L)=V_n(K)+\sum_{j=1}^n\lambda^j\binom{n}{j}V(K[n-j],L[j])=\lambda^l\binom{n}{l}V(K[n-l],L[l])$$
and
\begin{equation}\label{second_pol}
V_n(\di K+\lambda \di L)=V_n(\di K)+\sum_{j=1}^n\lambda^j\binom{n}{j}V(\di K[n-j],\di L[j]).
\end{equation}
Since \eqref{ineq_from_VC} holds for every $\lambda>0$, the polynomial in \eqref{second_pol} is a monomial of degree $l$. By Theorem~\ref{mix_volumes_Schneider} this is possible only if $\dim \di K=l$ and $\dim \di L=n-l$, which proves the result.
\end{proof}

Throughout the rest of the paper we will use systematically the following special situation of  Proposition~\ref{general_facts}.

\begin{remark}\label{segment}
Let $n\geq 2$ and let $\di\in\MAdd$ satisfy (VC). Then  
the image of a non-degenerate segment under $\di$ is a non-degenerate segment.
\end{remark}

Next we show that every  $\di\in\MAdd$ satisfying (VC) induces in a natural way a map from $\Gr(n,k)$ into itself, for every $1\leq k\leq n-1$. For $E\in \Gr(n,k)$, the set of convex bodies contained in $E$ will denoted by $\K(E)$, i.e., 
$$
\K(E):=\{K\in\K^n\,:\,K\subset E\}.
$$

\begin{proposition}\label{wohl def} Let $\di\in\MAdd$ satisfy the (VC) condition. Let $0\leq k\leq n$ and let $E\in\Gr(n,k)$.
Then there exists $\pi^{\di}_k(E)\in\Gr(n,k)$ such that $\di(K)\subset \pi^{\di}_k(E)$ for all $K\in \K(E)$, i.e., ${\rm Im}\,\di|_{\K(E)}\subset \K(\pi^\di_k(E))$.
\end{proposition}

\begin{proof} 
For $k=0$, the result follows by Lemma~\ref{additive}, and for $k=n$, it is trivial. 
Let $k\in\{1,\dots,n-1\}$ and let $E\in\Gr(n,k)$. Let $K$ and $L$ be $k$-dimensional convex bodies in $\K(E)$. We know, by Proposition~\ref{general_facts}, that $\di(K)$ and $\di(L)$ are $k$-dimensional as well.

We prove
that $\mathrm{span}(\di(K))=\mathrm{span}(\di(L))$.  For, let
$F:=\mathrm{span}(\di(K))+\mathrm{span}(\di(L))$ 
and assume $m=\dim(F)\geq k+1$. 
Let $M\in\K(F^{\perp})$ with $\dim M=\dim F^{\perp}=n-\dim F$.
Since $K,L\subseteq E$ with $\dim K=\dim L=k$ and $m\geq k+1$, we have that $\dim M\leq n-k-1$, which yields
\begin{equation}\label{eq_KLM}
0=\vol((K+L)+M).
\end{equation}
On the other hand, Proposition~\ref{general_facts} and the Minkowski additivity of $\di$ ensure that
\[
\vol\left(\di(K+L+M)\right)=\vol((\di(K)+\di(L))+\di(M)),
\]
which together with the assumptions $\dim F\geq k+1$ and $\dim M=n-\dim F$ yields
\[
0<\vol\left(\di(K+L+M)\right).
\]
But this contradicts~\eqref{eq_KLM}.

To finish the proof, we consider $K\in\K(E)$ with $\dim(K)<k$ and show that $\di(K)\subset\pi_k(E)$. For that, take $L\in\K(E)$ with dimension $k$. Then $K+L\subset E$ has also dimension 
$k$ so that $\di(K+L)=\di(K)+\di(L)\subset\pi_k(E)$, which implies $\di(K)\subset\pi_k(E)$.
\end{proof}

Using the above result, we define the following map.
\begin{definition}\label{def pi_k}
Let $\di\in\MAdd$ satisfy the (VC) condition and let $k\in\{0,\dots,n\}$. The map
\[
\func{\pi_k^{\di}}{\Gr(n,k)}{\Gr(n,k)}{E}{\pi^{\di}_k(E)}
\]
is well-defined and for every $E\in\Gr(n,k)$,
$$
\di(K)\subset\pi_k(E),\quad\forall\, K\in\K(E).
$$
\end{definition}
From now on we will omit the subscript $\di$ in $\pi^{\di}_k$, unless it is not clear from the context.

The additivity of $\di$ and the definition of $\pi_k$ imply the following additivity property of 
the functions $\pi_k$. 
\begin{lemma}\label{additivity of pi} 
Let $E_i\in\Gr(n,k_i)$, $k_i\in\{0,\dots,n\}$,
$i=1,2$, and let $j\in\{0,\dots,n\}$ be the dimension of $E_1+E_2$. Then
$$
\pi_{k_1}(E_1)+\pi_{k_2}(E_2)=\pi_j(E_1+E_2).
$$
\end{lemma}

\bigskip
Next, we introduce another map defined on $\Gr(n,1)$, which, together with $\pi_1$ will play the role of a Klain function of $\di$ (see Proposition~\ref{characterization}). 
\begin{definition}[Definition and Remark]\label{def rho1}
Let $\di\in\MAdd$ satisfy the (VC) condition. The map  
$$\func{\rho_1}{\Gr(n,1)}{(0,\infty)}{E}{V_1(\di(E\cap\ball))}$$
is well-defined. Further, it is continuous and strictly positive.
\end{definition}
Notice that Remark~\ref{segment} ensures that for any $E\in\Gr(n,1)$, $\dim \di(E\cap B^n)=1$, and $\rho_1$ is strictly positive. Finally, the continuity of $\di$ and $V_1$ yield that $\rho_1$ is continuous.

This real-valued map measures how $\di$ stretches the 1-dimensional volume on each line of $\R^n$ passing through the origin. 

\begin{definition}
Let $\di\in\MAdd$ satisfy the (VC) condition and let $\rho_1$ be associated to $\di$. The
following magnitudes are well-defined:
$$
m_\di:=\min_{\sfe}\rho_1,\quad M_\di:=\max_{\sfe}\rho_1.
$$
Further, $m_\di,M_\di>0$. 
\end{definition}

We aim to prove the following result: 

\begin{theorem}\label{teo Grassmannian} Let $n\geq 2$ and let $\di\in\MAdd$  
satisfy (VC). For every 
$k\in\{1,\dots,n-1\}$, the map $\pi_k:\Gr(n,k)\longrightarrow\Gr(n,k)$ 
given by Definition~\ref{def pi_k} is
\begin{enumerate}
\item[\emph{(i)}] continuous,
\item[\emph{(ii)}] bijective,
\item[\emph{(iii)}] bi-Lipschitz, i.e., there exist constants $c,C>0$ such that
$$
c\, d(E,F)\le d(\pi_k(E),\pi_k(F))\le C\, d(E,F),\quad\forall\, E,F\in\Gr(n,k).
$$

\end{enumerate}
\end{theorem}
The proof of the above results combines geometric and analytical tools, and will allow us to obtain a more precise knowledge of the induced action of a $\di\in\MAdd^{s,+}$ that satisfies (VC), on the Grassmannians. 
 
 In order to prove the above result we need the following lemma.
 
\begin{lemma}\label{geometry} Let $E_1,E_2\in\Gr(n,1)$, and let $v_1,v_2\in\sfe$ be such that
$$
[-v_i,v_i]=E_i\cap\ball,\quad i=1,2,\quad 1>\langle v_1,v_2\rangle\ge0.
$$
Let $\alpha\in[0,\pi/2]$ be such that $\cos(\alpha)=\langle v_1,v_2\rangle$. Then,
\begin{enumerate}
\item[\emph{(i)}] $d(E_1,E_2)=2\sin\left(\frac\alpha2\right)$, 
\item[\emph{(ii)}] $V_2([-v_1,v_1]+[-v_2,v_2])=4\,\sin(\alpha).$
\end{enumerate}
Moreover, the following inequality holds:
\begin{equation}\label{pink star}
\frac{1}{4}\le\frac{d(E_1,E_2)}{V_2((E_1\cap\ball)+(E_2\cap\ball))}\le\frac{\sqrt2}{4}.
\end{equation}
\end{lemma}

\begin{proof}
Item (i) is equation~\eqref{dist_Grn1}. 
To prove (ii) we use that $[-v_1,v_1]+[-v_2,v_2]$ is the union of four parallelepipeds whose 2-dimensional interiors are non-intersecting and of area equal
to the area of $[0,v_1]+[0,v_2]$, which is $\sin(\alpha)$.
In order to prove the inequality \eqref{pink star}, it is enough to use (i) and (ii) with $\alpha\in[0,\pi/2]$.
\end{proof}

\bigskip

\begin{proof}[Proof of Theorem~\ref{teo Grassmannian}]
Let $\di\in\MAdd$ satisfy (VC).
Let $\pi_k\,:\,\Gr(n,k)\longrightarrow\Gr(n,k)$ be given in Definition~\ref{def pi_k}, for $1\leq k\leq n-1$. 

Observing that for $E\in\Gr(n,k)$, the map $\pi_k$ is defined so that $\di(K)\subset\pi_k(E)$ for all $K\in\K(E),$ its continuity follows from the continuity of $\di$ and the definition of distance on $\Gr(n,k)$ (see \eqref{d_Gr}).

The injectivity follows from Proposition~\ref{general_facts}. Indeed, if $E_1\neq E_2\in\Gr(n,k)$ and $K_i\in E_i$ with $\dim K_i=k$,  $i=1,2$, we have $\dim(K_1+K_2)=m>k$.  As $\di$ is Minkowski additive, we also have $\dim(\di(K_1+K_2))=\dim(\di(K_1)+\di(K_2))>k$, which is only possible if $\spa(\di(K_1))\neq \spa(\di(K_2))$, i.e., $\pi_k(E_1)\neq\pi_k(E_2)$.

Next we prove (iii). We start with the case $k=1$. Let $E_1,E_2\in\Gr(n,1)$ be distinct and let $E=E_1+E_2\in\Gr(n,2)$ and $F=E^\perp$.
Denote
$$
S_i=E_i\cap\ball,\quad i=1,2,
$$
and let $L:=F\cap B^n\in \K^{n-2}(F)$. Since $E$ and $F$ are orthogonal, we have 
$$
V_n(S_1+S_2+L)=V_2(S_1+S_2)V_{n-2}(L),
$$
and Lemma~\ref{additivity of pi} yields
$$
\pi_2(E)\cap\pi_{n-2}(F)=0.
$$ 
By the Minkowski additivity of $\di$ and  \eqref{eq_Gardner},
\begin{eqnarray*}
\vol(\di(S_1+S_2+L))&=&\vol(\di(S_1)+\di(S_2)+\di(L))
\\
&=&c(\pi_2(E),\pi_{n-2}(F)) V_2(\di(S_1)+\di(S_2))\,V_{n-2}(\di(L)).
\end{eqnarray*}
The constant $c(\pi_2(E),\pi_{n-2}(F))>0$ depends continuously on $\pi_2(E)$ and $\pi_{n-2}(F)$ and then, ultimately, depends only on $E_1$ and $E_2$. We define
\begin{align*}
\gamma_{\di}&:=\inf_{\stackrel{E_1,E_2\in\Gr(n,1)}{E_1\neq E_2}}c(\pi_{2}(E_1+E_2),\pi_{n-2}((E_1+E_2)^{\perp})),
\\ \Gamma_{\di}&:=\sup_{\stackrel{E_1,E_2\in\Gr(n,1)}{E_1\neq E_2}}c(\pi_{2}(E_1+E_2),\pi_{n-2}((E_1+E_2)^{\perp})).
\end{align*} 
These are strictly positive and finite constants since they coincide, respectively, with
$$\inf_{F\in\Gr(n,n-2)}c(\pi_{2}(F^{\perp}),\pi_{n-2}(F))\quad\textrm{and}\quad\sup_{F\in\Gr(n,n-2)}c(\pi_{2}(F^{\perp}),\pi_{n-2}(F)).$$ 
Indeed, since the functions $\pi_2$, $\pi_{n-2}$, and $c$ are continuous and $\Gr(n,n-2)$ is compact, the above infimum and supremum are attained. 

Similarly, the values 
$$\lambda_{\di}:=\inf_{\stackrel{E_1,E_2\in\Gr(n,1)}{E_1\neq E_2}}V_{n-2}(\di((E_1+E_2)\cap B^n)),\quad \Lambda_{\di}:=\sup_{\stackrel{E_1,E_2\in\Gr(n,1)}{E_1\neq E_2}}V_{n-2}(\di(((E_1+E_2)^{\perp}\cap B^n))$$
are strictly positive and finite constants since 
$$\lambda_{\di}=\inf_{F\in\Gr(n,n-2)}V_{n-2}(\di(F\cap B^n))\quad\textrm{and}\quad\Lambda_{\di}=\sup_{F\in\Gr(n,n-2)}V_{n-2}(\di(F\cap B^n)).$$ 

On the other hand, by the (VC) condition, we have that
$$
0<c_\di\le\frac{V_n(\di(S_1+S_2+L))}{V_n(S_1+S_2+L)}\le C_\di
$$
where $c_\di$ and $C_\di$ (from the definition of (VC)) depend only on $\di$. 
Thus, we obtain that 
\begin{equation}\label{ineq1}
0<c_1\le\frac{V_2(\di(S_1)+\di(S_2))}{V_2(S_1+S_2)}\le c_2,
\end{equation}
where $c_1=\frac{c_\di}{\Gamma_{\di}\Lambda_{\di}}$ and $c_2=\frac{C_\di}{\gamma_{\di}\lambda_{\di}}$. 
Taking Lemma~\ref{geometry} into account we know that
\begin{equation}\label{ineq2}
2\sqrt{2}\,d(E_1,E_2)\le
V_2(S_1+S_2)\le 4\, d(E_1,E_2).
\end{equation}
We next obtain a similar bound for $V_2(\di(S_1)+\di(S_2))$. For that we apply Lemma~\ref{geometry} to the unit segments $\tilde S_i:=\pi_1(E_i)\cap B^n$, $i=1,2$. Up to a translation, we have $\tilde S_i=\frac{1}{\rho_1(E_i)}\di(S_i)$. 
We obtain, 
$$V_2(\di(S_1)+\di(S_2))=\rho_1(E_1)\rho_1(E_2)V_2\left(\frac{\di(S_1)}{\rho_1(S_1)}+\frac{\di(S_2)}{\rho_1(S_2)}\right),$$ 
and
\begin{equation}\label{ineq3}
2\sqrt{2}\,m_\di^2\, d(\pi_1(E_1),\pi_1(E_2))\le
V_2(\di(S_1)+\di(S_2))\le 4 M_\di^2\, d(\pi_1(E_1),\pi_1(E_2)).
\end{equation}
Therefore, by~\eqref{ineq1}, \eqref{ineq2} and \eqref{ineq3}, we have 
$$
0<k_1\le\frac{d(\pi_1(E_1),\pi_1(E_2))}{d(E_1,E_2)}\le k_2,
$$
where $k_1=\frac{c_\di}{\sqrt{2}M_{\di}^2\Gamma_{\di}\Lambda_{\di}}$ and $k_2=\frac{C_{\di}\sqrt{2}}{m_{\di}^2\gamma_{\di}\lambda_{\di}}$ are constants independent of $E_1$ and $E_2$. Hence, $\pi_1$ is bi-Lipschitz. 

\medskip
We next prove that $\pi_1$ is surjective, which finishes the proof of (ii) for $k=1$. Let ${\mathcal I}\subset\Gr(n,1)$ be the image of $\pi_1$, i.e.,
$$
{\mathcal I}:=\{E\in\Gr(n,1)\,:\,\exists\, E'\in\Gr(n,1)\,\mbox{ s.t. }\, \pi_1(E')=E\}.
$$
Let $E_1,E_2\in\Gr(n,1)$, $E_1\ne E_2$,
and let $F=E_1+E_2\in\Gr(n,2)$. We know from Definition~\ref{def pi_k} that $\dim F=\dim\pi_2(F)=2$. 
Further, Lemma~\ref{additivity of pi} ensures that $\pi_1(E)\in\Gr(\pi_2(F),1)$ for every $E\in\Gr(F,1)$. 
Due to the bi-Lipschitz property of $\pi_1$, the set
$$
\pi_1(\Gr(F,1)):=\{\pi_1(E)\,:\,E\in\Gr(F,1)\}
$$ 
is a (non-empty) closed and open subset of $\Gr(\pi_2(F),1)$. Indeed, if $(E_i)_i$ is a convergent sequence in $\pi_1(\Gr(F,1))$, then, since $\pi_1$ is bi-Lipschitz, the sequence $\tilde E_i$ given by $E_i=\pi_1(\tilde E_i)$ is convergent  to $\tilde E\in\Gr(F,1)$. By definition, $\pi_1(\tilde E)\in\pi_1(\Gr(F,1))$ and by continuity, it is the limit of the sequence $(E_i)$. Thus, the set $\pi_1(\Gr(F,1))$ is closed. To prove that it is open, we use that $\Gr(F,1)$ can be identified with the unit sphere $\sfed$ in the plane, with the convention that antipodal points are identified (see Subsection~\ref{grass}). In turn $\sfed$, with the identification of antipodal points, can be identified with the interval $[0,\pi)$. The topology given to $\Gr(F,1)$, under these identifications, is equivalent to the usual topology of the real line, restricted to $[0,\pi)$.  
Now, a bi-Lipschitz map from $[0,\pi)$ onto itself takes open sets into open sets, since it takes neighborhoods into neighborhoods.
Hence, $\pi_1(\Gr(F,1))$ is open.

Now, since the set $\Gr(\pi_2(F),1)$ is connected, we have
\begin{equation}\label{pi_1 surj}
\pi_1(\Gr(F,1))=\Gr(\pi_2(F),1).
\end{equation}
Combining \eqref{pi_1 surj} and 
$$
\pi_2(F)=\pi_1(E_1)+\pi_1(E_2),
$$
which follows from Lemma~\ref{additivity of pi}, we obtain
$$
{\mathcal I}\supset\Gr(\pi_1(E_1)+\pi_1(E_2),1),\quad\forall\, E_1,E_2\in\Gr(n,1).
$$
This property can be inductively extended to any finite sum of elements of $\Gr(n,1)$. 

Let $\{e_1,\dots,e_n\}$ denote the standard basis of $\R^n$. If we choose $E_i=\spa(e_i)$, $i=1,\dots,n$, and use this property together with Proposition~\ref{general_facts}, we obtain ${\mathcal I}=\Gr(n,1)$, i.e., $\pi_1$ is surjective.

\medskip
We next use the surjectivity of $\pi_1$ to obtain that $\pi_k$, for $k\ge2$, is also surjective. Indeed, let $F\in\Gr(n,k)$ and let $\{v_1,\dots,v_k\}$ be a basis of $F$. Let $E_i=\spa(v_i)$, $i=1,\dots,k$. The surjectivity of $\pi_1$ yields that there exist $F_i\in\Gr(n,1)$ such that  
$F_i=\pi_1^{-1}(E_i)$, $i=1,\dots,k$; by Lemma~\ref{additivity of pi} we obtain
$$
\pi_k(F_1+\dots+F_k)= F.
$$
\medskip
We conclude the proof of Theorem~\ref{teo Grassmannian} by proving that $\pi_k$ is bi-Lipschitz, for $k\ge2$. 

Let $F_1,F_2\in\Gr(n,k)$. We recall that, by \eqref{dist_angles_Grnk}, we have
$$
d(F_1,F_2)=\max\left\{
\max_{L_1\in\Gr(F_1,1)}\left(\min_{L_2\in\Gr(F_2,1)}d(L_1,L_2)\right),
\max_{L_2\in\Gr(F_2,1)}\left(\min_{L_1\in\Gr(F_1,1)}d(L_1,L_2)\right)
\right\}.
$$

Let $\tilde L_1\in\Gr(F_1,1)$ be fixed and $\tilde L_2\in\Gr(F_2,1)$ be such that 
$$
d(\tilde L_1,\tilde L_2)=\min_{L_2\in\Gr(F_2,1)}d(\tilde L_1,L_2).
$$
Then, using that $\pi_1$ is bi-Lipschitz, we have
\begin{eqnarray*}
\min_{L_2\in\Gr(F_2,1)}d(\tilde L_1,L_2)&=&d(\tilde L_1,\tilde L_2)\\
&\ge&c\, d(\pi_1(\tilde L_1),\pi_1(\tilde L_2))\\
&\ge&c\,\min_{L'_2\in\Gr(\pi_k(F_2),1)} d(\pi_1(\tilde L_1),L'_2),
\end{eqnarray*}
for some constant $c>0$. Thus, taking maximum with respect to $L_1\in\Gr(F_1,1)$ and using the surjectivity of $\pi_1$, we get
$$
\max_{L_1\in\Gr(F_1,1)}\left(\min_{L_2\in\Gr(F_2,1)}d(L_1,L_2)\right)\ge
c\, \max_{L'_1\in\Gr(\pi_k(F_1),1)}\left(\min_{L'_2\in\Gr(\pi_k(F_2),1)}d(L'_1,L'_2)\right).
$$ 
In the same way, the corresponding inequality where the roles of $F_1$ and $F_2$ are interchanged is obtained. From these and the above definition of Hausdorff metric, we have
$$
d(F_1,F_2)\ge c\, d(\pi_k(F_1),\pi_k(F_2)).
$$
The proof is completed applying the same argument to $\pi_k^{-1}$, to obtain the reverse inequality. 
\end{proof}

Next, we collect some special assertions of the case $k=1$. 

\begin{lemma}\label{pi_1 interpretation} Let $\di\in\MAdd$ satisfy (VC) and let $\pi_1\,:\,\Gr(n,1)\longrightarrow\Gr(n,1)$ be the map in Definition \ref{def pi_k}. Let $S$ be a segment in $\R^n$.
\begin{enumerate}
\item[\emph{(i)}] If $E=\spa S\in \Gr(n,1)$, then
$
\di(S)\subset\pi_1(E).
$
\item[\emph{(ii)}]For every $E\in\Gr(n,1)$ there exists a translation $t(E)\in\pi_1(E)$ for which
$$\di(E\cap\ball)=\frac12\rho_1(E)(\pi_1(E)\cap\ball)+t(E).$$
If $\di$ is also an $o$-symmetrization, then
\begin{equation}\label{identity}
\di(E\cap\ball)=\frac12\rho_1(E)(\pi_1(E)\cap\ball),\quad\forall E\in\Gr(n,1)
\end{equation}
\item[\emph{(iii)}] The pair $(\pi_1,\rho_1)$ identifies the image through $\di\in\MAdd^s$ of any centered segment having length 2.
The 1-homogeneity and translation invariance yield that the image of any segment in $\R^n$ can be deduced. 
\end{enumerate}
\end{lemma}

This allows us to interpret
$(\pi_1,\rho_1)$ as a Klain map of $\di\in\MAdd$ (cf. Theorem~\ref{klain_injectivity}), when $\di$ is also even and $o$-symmetrization. We first compute the Klain function of the real-valued valuations associated to $\di$, given by \eqref{white star}, namely, $\rvd_u=h(\di(\cdot),u)$, for any $u\in\R^n$.

\begin{proposition}\label{Klain map} Let $\di\in\MAdd^s$ satisfy (VC). For $u\in\R^n$, 
the Klain map of $\rvd_u=h(\di(\cdot),u)$ is determined by the pair $(\pi_1,\rho_1)$. In particular,
$$
\Kl_{\rvd_u}(E)=\frac12 \rho_1(E)\,V_1([-u,u]|\pi_1(E)),
$$
for all $E\in\Gr(n,1)$.
\end{proposition}
\begin{proof} Using the definition of $\rvd_u$, \eqref{identity 2}, and \eqref{identity}, we get
\begin{eqnarray*}
\Kl_{\rvd_u}(E)&=&\frac12\, h(\di(E\cap\ball),u)\\
&=&\frac14\,\rho_1(E)h(\pi_1(E)\cap\ball,u)\\
&=&\frac14\,\rho_1(E)\,|\langle v,u\rangle|,
\end{eqnarray*}
where $v\in\sfe$ is such that
$$
\pi_1(E)\cap\ball=[-v,v].
$$
The proof is concluded observing that, by \eqref{eq_Gardner},
$$
|\langle u,v\rangle|=2V_1([-u,u]|\pi_1(E)).
$$
\end{proof}

\begin{proposition}\label{characterization} An operator $\di\in\MAdd^{s,+}$ satisfying (VC) is uniquely determined by the pair $(\pi_1,\rho_1)$. 
\end{proposition}

\begin{proof}
Notice, first, that as $\di$ is even, then $\rvd_u$ is even for every $u\in\R^n$. As a consequence of Theorem~\ref{klain_injectivity} and Proposition~\ref{Klain map},  
$\rvd_u$ is uniquely determined by $(\pi_1,\rho_1)$ for every $u$, i.e., the support function of $\di(K)$ is uniquely determined by $(\pi_1,\rho_1)$, for every $K$.
\end{proof}

\section{Representation formulae}\label{s: representation formulae}
In this section we prove the representation result contained in Theorem~\ref{thm n=2}, as well as its extension to higher dimension in Theorem~\ref{thm n ge 2}. These representation formulae will be useful in the coming section, in particular to obtain the characterization result for the difference body given in Theorem~\ref{+mon intro}.

\begin{theorem}\label{thm n=2} 
Let $n=2$. An operator $\di\in\MAdd^{s,+}$ satisfies (VC) if and only if there exist
$\rho\,:\,\Gr(2,1)\longrightarrow\R$, continuous and strictly positive, and $\pi\,:\,\Gr(2,1)\longrightarrow\Gr(2,1)$, bijective and bi-Lipschitz, such that for every $K\in\K^2$,
$$h(\di(K),u)=\int_{\sfed}
\rho(\v) V_1([-u,u]|\pi(\v))
dS_1(K,v),\quad\forall\, u\in\sfed.$$
\end{theorem}

\begin{proof}
We recall that the notation $\x\in\R^2$ stands for the orthogonal vector to $x=(x_1,x_2)\in\R^2\setminus\{0\}$ given by $(-x_2,x_1)$.

Let $\pi_1$ and $\rho_1$ be the functions, associated to $\di$, given by Definitions~\ref{def pi_k}  and~\ref{def rho1}.
Let $u\in\sfe$. We consider the map $\rv_u\,:\,\K^2\longrightarrow\R$ defined as
\begin{equation}\label{eq: rv}
\rv_u(K)=\int_{\sfed}\rho_1(\v) V_1([-u,u]|\pi_1(\v))dS_1(K,v).
\end{equation}
Since $S_1$ is a 1-homogeneous and translation invariant valuation (see Section~\ref{s: mixed vol}), it follows that $\rv_u$ is also a 1-homogeneous and translation invariant valuation. Moreover, as $S_1(K,\cdot)$ is weakly continuous with respect to $K\in\K^2$, $\rv_u$ is continuous and it is also even, since $S_1(-K,v)=S_1(K,-v)$ for any $v\in\sfed$. 

We prove next that 
\begin{equation}\label{claim}
\rv_u\equiv\rvd_u,
\end{equation}
where $\rvd_u=h(\di(\cdot),u)$. Equality \eqref{claim} will follow if we prove that
$\rv_u$ and $\rvd_u$ have the same Klain function (cf. Theorem~\ref{klain_injectivity}). In order to determine $\Kl_{\rv_u}$, let
$E\in\Gr(2,1)$ and $v_0\in\sfed$ be such that $E=\mathrm{span}(\vo)$. Then
$$
\Kl_{\rv_u}(E)=\frac12\,\rv_u(E\cap B^2)=\frac12\,\rv_u([-\vo,\vo]).
$$
The surface area measure of order 1 of the segment $[-\vo,\vo]$ is
\begin{equation}\label{area_segment}
S_1([-\vo,\vo],\cdot)=2\,
\left(
\delta_{v_0}(\cdot)+\delta_{-v_0}(\cdot)
\right),
\end{equation}
where $\delta_v$ denotes the Dirac measure concentrated at $v\in\sfed$. Thus, from \eqref{eq: rv} and \eqref{area_segment} we have
$$
\Kl_\rv(E)=\frac12\rho_1(E)\,V_1([-u,u]|\pi_1(E)).
$$
The conclusion follows from Proposition~\ref{Klain map}. 

\medskip

Next we prove the ``only if'' part of Theorem~\ref{thm n=2}. Let
$\rho\,:\,\Gr(2,1)\longrightarrow\R$ be continuous and strictly positive, and let $\pi\,:\,\Gr(2,1)\longrightarrow\Gr(2,1)$ be bijective, bi-Lipschitz, and such that for every $K\in\K^2$, 
\[
h(\di(K),u)=\int_{\sfed}
\rho(\v) V_1([-u,u]|\pi(\v))
dS_1(K,v),\quad\forall\, u\in\sfed.
\]
For every $K\in\K^2$, the map 
\[
\func{h}{\R^2}{\R}{u}{\displaystyle\int_{\sfed}\rho(\v) V_1([-u,u]|\pi(\v))dS_1(K,v)}
\]
is convex and 1-homogeneous.
Therefore, there exists a convex body $\di(K)\in\K^2$ such that $h(u)=h(\di(K),u)$ for any $u\in\R^2$. 
The operator $\di:\K^2\longrightarrow \K^2$, defined by $K\mapsto\di(K)$, with 
$$
h(\di(K),u)=\displaystyle\int_{\sfed}\rho(\v) V_1([-u,u]|\pi(\v))dS_1(K,v), 
$$
as argued with \eqref{eq: rv},  is an even $o$-symmetrization and $\di\in\MAdd$. It remains to prove that it satisfies (VC).

In order to prove (VC) for $\di$ we show first that it is satisfied for parallelograms.
Let $S_1,S_2$ be (non-degenerated) segments in $\R^2$. We prove that there exist $c_1,c_2>0$, not depending on $S_1$ and $S_2$, such that
\begin{equation}\label{VC_parall}
c_1\,V_2(S_1+S_2)\le V_2(\di(S_1+S_2))\le c_2\, V_2(S_1+S_2).
\end{equation}
Indeed, using the translation invariance of the operator $\di$, we can assume w.l.o.g. that there are  $v_1,v_2\in\sfed$ and $\alpha_1,\alpha_2>0$ such that $S_1=\alpha_1[-v_1,v_1]=\alpha_1S_{v_1}$ and $S_2=\alpha_2[-v_2,v_2]=\alpha_2S_{v_2}$. 
Therefore we have
\begin{equation}\label{area}
V_2(\alpha_1S_{v_1}+\alpha_2S_{v_2})=\alpha_1\alpha_2V_2(S_{v_1}+S_{v_2}).
\end{equation}
If we denote $\rho_{max}=\textrm{max}_{E\in\Gr(2,1)}\rho(E)$ and use the additivity of $\di$, \eqref{area}, \eqref{identity}, \eqref{pink star}, the Lipschitz property of $\pi_1$, and again~\eqref{pink star} and \eqref{area}, we obtain
\begin{align*}
V_2(\di(S_1+S_2))&=V_2(\di(\alpha_1S_{v_1}+\alpha_2S_{v_2}))=V_2(\alpha_1\di(S_{v_1})+\alpha_2\di(S_{v_2}))
\\&=\frac14\alpha_1\alpha_2V_2(\rho(E_{v_1})\pi_1(E_{v_1}\cap B^1)+\rho(E_{v_2})\pi_1(E_{v_2}\cap B^1))
\\&\leq \frac14\alpha_1\alpha_2\rho_{{max}}^2V_2(\pi_1(E_{v_1}\cap B^1)+\pi_1(E_{v_2}\cap B^1))
\\&\leq \alpha_1\alpha_2\rho_{{max}}^2 d(\pi_1(E_{v_1}),\pi_1(E_{v_2}))
\\&\leq \alpha_1\alpha_2\rho_{{max}}^2C d(E_{v_1},E_{v_2})
\leq \frac14\alpha_1\alpha_2\rho_{{max}}^2C\sqrt{2} V_2(E_{v_1}\cap B^1+E_{v_2}\cap B^1)
\\&=\frac{\sqrt{2}}{4}C\rho_{{max}}^2V_2(S_1+S_2),
\end{align*}
for an appropriate constant $C$ depending only on the Lipschitz property of $\pi_1$. 
Notice that both $C$ and $\rho_{max}$ do not depend on $S_1,S_2$.
With an analogue argument we prove that 
$$2cV_2(S_1+S_2)\leq V_2(\di(S_1+S_2))$$ for a constant $c>0$, which is independent of $S_1,S_2$. Hence, we have proved that $\di$ satisfies (VC) for parallelograms. 

Next, we show that $\di$ satisfies (VC) for zonotopes. Let $m\in\N$ and let $S_1,\dots,S_m$ be segments. Consider the zonotope $K=S_1+\dots+S_m$. Using the multilinearity of mixed volumes, together with Theorem~\ref{mix_volumes_Schneider}, we have
$$V_2(S_1+\dots+S_m)=2\sum_{j<l}V(S_j,S_l)=\sum_{j<l}V_2(S_j+S_l).$$
Thus, using \eqref{VC_parall}, 
\begin{align*}
c_1V_2(S_1+\dots+S_m)&=c_1\sum_{j<l}V_2(S_j+S_l)
\\&\leq\sum_{j<l}V_2(\di(S_j+S_l))=\sum_{j<l}V_2(\di S_j+\di S_l)
\\&=V_2(\di S_1+\dots+\di S_m)=V_2(\di(S_1+\dots+S_m)).
\end{align*}
Analogously, we obtain that
$$V_2(\di (S_1+\dots+S_m))\leq c_2V_2(S_1+\dots+S_m).$$
Hence, $\di$ satisfies (VC) when restricted to the family of zonotopes. 
The same assertion follows for zonoids, since $\di$ and $V_2$ are continuous. Using that in dimension $n=~2$ all centrally symmetric bodies are zonoids (see \cite[Corollary~3.5.7]{schneider.book14}), we obtain the statement for all symmetric bodies, i.e., 
\begin{equation}\label{baby blue star}
c_1\,V_2(K)\le V_2(\di K)\le c_2\, V_2(K),
\end{equation}
for all $K\in\K^2_s$.
To prove that $\di$ satisfies (VC) for an arbitrary convex body $K$, not necessarily centrally symmetric, we note that since $\di$ is even and additive, 
\begin{equation}\label{dark star}
\di K=\frac12\,\di(K+(-K))=\frac12\di(DK),
\end{equation}
for all $K\in\K^2$.
The proof is concluded using \eqref{baby blue star}, \eqref{dark star}, and the Rogers-Shephard inequality \eqref{RS} (cf. Lemma~\ref{composition_D}).
\end{proof}

We obtain a representation result in higher dimension using Theorem~\ref{6.4.12}.

\begin{theorem}\label{thm n ge 2} Let $\di\in\MAdd^{s,+}$ satisfy the (VC) condition.
Then there exist $\rho_1\,:\,\Gr(n,1)\longrightarrow\R$, continuous and strictly positive, and $\pi_1\,:\,\Gr(n,1)\longrightarrow\Gr(n,1)$, bijective and
bi-Lipschitz, such that for every generalized zonoid $Z$, with generating measure $\rho_Z$, we have
\begin{equation}\label{eq: rv zonoids}
h(\di Z,u)=\frac12\,\int_{\sfe}
\rho_1(E_v)\,V_1([-u,u]|\pi_1(E_v))\,d\rho_Z(v),\quad\forall\, u\in\R^n,
\end{equation}
where for $v\in\sfe$, $E_v=\spa(v)$. 
\end{theorem}

\begin{proof}
The proof follows the lines of the first part of the proof of Theorem~\ref{thm n=2}, where to define $\rv_u$ we use \eqref{eq: rv zonoids} instead of \eqref{eq: rv} and obtain the conclusion from Theorem~\ref{6.4.12}, instead of Proposition~\ref{Klain map}.
\end{proof}

\subsection{Encoding the image of segments: the function $p$}

We have seen in Theorem~\ref{characterization} that $\di\in\MAdd^{s,+}$ determines (and is determined by) the maps $\pi_1$ and $\rho_1$, both defined on $\Gr(n,1)$,
which completely describe the action of $\di$ on segments. It is tempting to think about a possible map $\R^n\longrightarrow \R^n$, with $v\mapsto w$ such that $\di([-v,v])=[-w,w]$. It is easy to realize that such an association is not well-defined. However, it would be practical to have a unique function defined on $\R^n$, which provides the same information as the pair $(\pi_1,\rho_1)$.  
For $v\in\sfe$, let $S_v=[-v,v]$ be the segment joining $-v$ and $v$.
From Proposition~\ref{general_facts}, for every $v\in\sfe$, there exists $w\in\R^n$, $w\ne 0$, such that
$$
\di([-v,v])=\di(S_v)=[-w,w]=S_w.
$$
Clearly, $w$ is not uniquely determined, as we may replace it by $-w$. Taking this into account, we define a function, in the above spirit, which enjoys several useful properties.

\begin{proposition}[and Definition]\label{p} Let $\di\in\MAdd^{s,+}$ satisfy (VC) and let $e\in\sfe$. There exists a measurable function 
$p\,:\,\sfe\longrightarrow\R^n\setminus\{0\}$ such that
\begin{enumerate}
\item[\emph{(i)}]
$$\di S_v=S_{p(v)},\quad\forall\, v\in\sfe;$$
\item[\emph{(ii)}]
$$
p(v)=p(-v),\quad
\langle p(v),e\rangle\ge 0,\quad\forall\,v\in\sfe.
$$
\end{enumerate}
Moreover, $p$ is continuous at every $v$ such that $\langle p(v),e\rangle\ne0$.
\end{proposition} 

\begin{proof} 
Let $\di\in\MAdd^{s,+}$ satisfy (VC). 
Without loss of generality we may assume that $e=e_n=(0,\dots,0,1)$. For $k\in\{1,\dots,n\}$ we set
$$
C_k=\{(x_1,\dots,x_n)\in\sfe\,:\,x_k>0,\, x_{k+1}=\dots=x_{n}=0\}.
$$
In particular,
$$
C_1=\{(1,0\dots,0)\} \quad \text{and} \quad
C_n=\{(x_1,\dots,x_n)\in\sfe\,:\,x_n>0\}.
$$
Note that $C_k\cap C_{k'}=\emptyset$ for $k\ne k'$. Let $\Omega$ be defined as follows
$$
\Omega=\bigcup_{k=1}^n C_k\subset\sfe.
$$
It is easy to check that for every $w\in\sfe$ we have either $w\in\Omega$ or $-w\in\Omega$. In other words, $\Omega$ contains exactly one of the points
$w$ and $-w$, for every $w\in\sfe$. In particular, for every $v\in\sfe$ there exists exactly one point $w\in\R^n\setminus\{0\}$ (by Proposition~\ref{general_facts}), that we will denote by $p(v)$, such that 
\begin{equation}\label{def_p}
\di S_v=\,S_{p(v)}\quad\mbox{and}
\quad\frac{p(v)}{\|p(v)\|}\in\Omega.
\end{equation}
In this way we have defined the map $p\,:\,\sfe\longrightarrow\R^n\setminus\{0\}$. We show next that it satisfies the other stated properties. First, since $\Omega\subset\{x_n\ge0\}$, 
we have 
$$
\langle p(v),e\rangle\ge0,\quad\forall\, v\in\sfe.
$$
Moreover, for every $v\in\sfe$ we have, by the first equality in the definition of $p$ in \eqref{def_p}, $p(-v)=\pm p(v)$. On the other hand, as $\Omega$ does not contain any pair of antipodal points,
we deduce $p(-v)=p(v)$.

We prove next the stated continuity property. Let $v\in\sfe$ be such that $\langle p(v),e\rangle>0$  and let $(v_i)_{i\in\N}$, be a sequence in $\sfe$, converging to $v$. By the continuity of $\di$, $S_{p(v_i)}$ converges to $S_{p(v)}$. 
If $w_i\in\R^n$ is such that 
$$ 
\di S_{v_i}=[-w_i,w_i],\quad\forall\, i\in\N,
$$
then we may assume, up to interchanging $w_i$ and $-w_i$, that $\langle w_i,e\rangle>0$ for $i$ sufficiently large. Then
$$
\lim_{i\to\infty}w_i=p(v)\quad\mbox{and}\quad w_i=p(v_i),\quad\forall\, i\in\N,
$$
which proves the continuity of $p$ at $v$, for every $v\in\sfe$ satisfying $\langle p(v),e\rangle>0$.

Finally, we prove the measurability of the application $p$ at $\sfe$. For $k=1,\dots,n$ we set
$$
D_k=p^{-1}(C_k)=\left\{v\in\sfe\,:\,\frac{p(v)}{\|p(v)\|}\in C_k\right\}.
$$
$D_k$ is measurable. Indeed, note that 
$$
D_1=\left\{v\in\sfe\,:\,\frac{p(v)}{\|p(v)\|}=(1,0,\dots,0)\right\}.
$$
By the continuity of $\di$ this set is closed, and hence measurable. We argue by induction.
Assume now that $D_1\cup\dots\cup D_k$ is closed; we prove that
$D_1\cup\dots\cup D_k\cup D_{k+1}$ is closed as well. Let $(v_i)_{i\in\N}$, be a convergent sequence in $D_1\cup\dots\cup D_k\cup D_{k+1}$. 
By the definition of the sets $D_j$, if $v\in D_j$, we have, for every $i\in\N$,
$$
p(v_i)=(p_1(v_i),\dots,p_{k+1}(v_i),0,\dots,0)\quad\mbox{with $p_{k+1}(v_i)\geq 0$.}
$$
First we observe that if, up to a subsequence, $p_{k+1}(v_i)=0$ for every $i\in\N$, then, the sequence $(v_i)_{i\in\N}$ lies in $D_1\cup\dots\cup D_k$ and the induction hypothesis yields the statement. Hence, we assume $p_{k+1}(v_i)>0$ for every $i\in\N$.
\\ Let
$$
 v=\lim_{i\to\infty} v_i.
$$
We want to prove that $v\in D_1\cup\cdots\cup D_k\cup D_{k+1}$, i.e., $p(v)=(p_1(v),\dots,p_{k+1}(v),0,\dots,0)$ 
where the last non-zero coordinate (not necessarily $p_{k+1}(v)$) is positive.

Next we show that, up to a subsequence, $(p(v_i))_i$ converges to a vector 
$$
w=(w_1,\dots,w_{k+1},0,\dots,0),
$$
as $i$ tends to infinity. Since $\sfe$ is compact and $v_i\in D_1\cup\cdots\cup D_k\cup D_{k+1}$, we have $\dfrac{p(v_i)}{\|p(v_i)\|}\in\sfe$, and hence, 
$\left(\frac{p(v_i)}{\|p(v_i)\|}\right)_{i\in\N}$ has a convergent subsequence in $\sfe$. Moreover, we claim that there is $C>0$ such that 
$\|p(u)\|\leq C$ for every $u\in\sfe$, and hence, $(p(v_i))_i$ converges, up to a subsequence, to a vector in $\R^n$. The claim follows from 
the continuity of $u\mapsto\|p(u)\|$, which is an immediate consequence of the continuity of $\di$ (see Proposition~\ref{continuity}) and the compactness of $\sfe$. Let $w=\lim_{i\to\infty} p(v_i)$.
By the continuity of 
$\di$, we have
$$
\di S_{v}=S_w.
$$
Hence, $p(v)=\pm w$. 
If $w_{k+1}>0$ (resp. $w_{k+1}<0)$, then, by the definition of $p$, $p(v)=w$ (resp. $p(v)=-w$), which implies that $v\in D_1\cup\dots\cup D_k\cup D_{k+1}$, since $p(v_i)=(p_1(v_i),\dots,p_{k+1}(v_i),0,\dots,0)$. 
If $w_{k+1}=0$, we can argue in the same manner with the last non-zero coordinate of $w$. We use the induction hypothesis to ensure that there is such a last non-zero coordinate in the limit, due to $D_1\cup\dots\cup D_k$  being closed. Similarly to the 
previous case, we have $p(v)=\pm w$ and 
$v\in D_1\cup\dots\cup D_k$. This concludes the proof that $D_1\cup\dots\cup D_k$ is closed for every $k$. As a consequence, 
$D_k$ is measurable for every $1\leq k\leq n$. 

The continuity of $\di$ and the definition of $D_k$ imply that $p$ is continuous on each $D_k$. Hence $p$ is measurable on each $D_k$ (a function which is continuous over a measurable set $A$ is measurable on $A$). As $\sfe$ is the disjoint union of $D_1,\dots,D_n$, $p$ is measurable on $\sfe$.
\end{proof}

\medskip

From now on, whenever we need the function $p$ associated to the operator $\di\in\MAdd^{s,+}$ from Proposition~\ref{p}, we will set $e\in\sfe$ fixed. We will use also $p$ to denote the 1-homogeneous extension of the function $p$ coming from Proposition~\ref{p} to $\R^n$, associated to the operator $\di\in\MAdd^{s,+}$. 

In the next, we collect some other important facts about the function $p$.
\begin{lemma}\label{facts about p}
Let $\di\in\MAdd^{s,+}$ and let $p$ be the one homogeneous extension of the function $p$ determined by Proposition~\ref{p}.
Let $E\in\Gr(n,1)$ and let $v,w\in\sfe$ be such that
$E\cap\sfe=\{\pm v\}$ and $\pi_1(E)\cap\sfe=\{\pm w\}$,
then
\begin{equation}\label{importante 2}
\rho_1(E)=2\|p(v)\|,\quad S_{p(v)}\subset\pi_1(E),
\quad\mbox{and}\quad
p(v)=\pm\rho_1(E)\,w.
\end{equation}
\end{lemma}

\begin{proof}
For the first equality, we observe that 
$\rho_1(E)=V_1(\di(E\cap B^n))=V_1(\di([-v,v]))=V_1([-p(v),p(v)])$. Since $\pi_1(E)=\spa\{[-w,w]\}=\spa\{[-p(v),p(v)]\}$, by the definition of $w$ and $p$, the second assertion follows. The last assertion follows from the definition of $\rho_1$ and the previous facts.
\end{proof}

\begin{proposition}\label{continuity} Let $\di\in\MAdd^{s,+}$ and let $p$ be given by Proposition~\ref{p}. 
Let $u\in\R^n$. The following maps are continuous:
\begin{enumerate}
\item[\emph{(i)}]
\begin{equation}\label{blue star}
v\,\mapsto\,|\langle p(v),u\rangle|,\quad v\in\sfe,
\end{equation}
\item[\emph{(ii)}]
\begin{equation}\label{red star}
v\,\mapsto\,\|p(v)\|,\quad v\in\sfe.
\end{equation}
\end{enumerate}
\end{proposition}

\begin{proof}
By the definition of $p$, we have $h(\di S_v,u)=|\langle p(v),u\rangle|$. Since $\di$ and the support function are continuous, the map in \eqref{blue star} is also continuous. 
Similarly, the map in \eqref{red star} is continuous since $\rho_1(E)=2\|p(v)\|$, by \eqref{importante 2}.
\end{proof}

Now, Theorems~\ref{thm n=2} and~\ref{thm n ge 2} can be rephrased in terms of the function $p$. 

\begin{corollary}\label{thm n=2 bis} 
Let $n=2$, and let $\di\in\MAdd^{s,+}$ satisfy (VC). Let $p$ be the function associated to $\di$ in Proposition~\ref{p}. Then, for every $K\in\K^2$,
$$
h(\di K,u)=\frac12\,\int_{\sfed}|\langle \pv,u\rangle|dS_1(K,v),\quad\forall\, u\in\R^2.
$$
\end{corollary}

\begin{corollary} Let $n\geq 2$ and $\di\in\MAdd^{s,+}$ satisfy (VC). Let $p$ be the function associated to $\di$ in Proposition~\ref{p}.
Then, for every generalized zonoid $Z$, with generating measure $\rho_Z$, we have
$$
h(\di Z,u)=2\,\int_{\sfe}|\langle p(v),u\rangle|\,d\rho_Z(v),\quad\forall\, u\in\R^n.
$$
\end{corollary}

\section{The monotonic case}\label{sec: 1h and mon}

In this section we provide the proof of Theorem~\ref{+mon intro}. 
We start with the following lemma, whose proof is straightforward.
\begin{lemma}\label{lemma linear} Let $\di\in\MAdd^{s,+}$ be  monotonic and satisfy the (VC) condition. Then, for every
$g\in\GL(n)$ the application $$\func{\di_g}{\K^n}{\K^n}{K}{g(\di K)}$$ is 
monotonic, $\di_g\in\MAdd$, and satisfies (VC). 
\end{lemma}

First we proof a weak version of Theorem~\ref{+mon intro}, where evenness and symmetrization properties are also assumed. 

\begin{proposition}\label{th monotone weak} Let $n\geq 2$. An operator $\di\in\MAdd^{s,+}$ satisfies (VC) and is monotonic increasing if and only if there is a $g\in\GL(n)$ such that $\di K=gDK$ for every $K\in\K^n$.
\end{proposition}

\begin{proof}
We start with the following reduction. 

Let $\{e_1,\dots,e_n\}$ be the canonical basis of $\R^n$ and let
$w_1,\dots,w_n\in\R^n$ be such that 
$$
\di([-e_i,e_i])=[-w_i,w_i]
$$
for all $i=\{1,\dots,n\}$.
Proposition~\ref{general_facts} ensures that $w_1,\dots,w_n$ are linearly independent. Hence there exists a unique $g\in\GL(n)$ such that $g(w_i)=e_i$ for every $i=1,\dots,n$. By Lemma
\ref{lemma linear}, the operator $\di_g$, defined by $\di_g(K)=g(\di K)$, has the properties of $\di$, namely, $\di_g\in\MAdd^{s,+}$, it satisfies (VC), and is monotonic increasing. Moreover,
$$
\di_g([-e_i,e_i])=[-e_i,e_i],\quad \forall\, i=1,\dots,n.
$$
In the rest of the proof we will 
work with $\di_g$ instead of $\di$ but write, for simplicity $\di$ instead of $\di_g$. Our aim is to 
prove that $\di K=DK$ for every $K\in\K^n$.

Let $p:\sfe\longrightarrow \R^n\setminus\{0\}$ be associated to $\di$ as indicated in Proposition~\ref{p} with the choice $e=e_n$. In particular we have 
$$
\langle p(v),e_n\rangle\ge0,\quad\forall\, v\in\sfe.
$$
Proposition~\ref{p} yields that $p$ is continuous at every $v_0$ such that $\langle p(v_0),e_n\rangle\ne0$. 
Using the above reduction (from $\di$ to $\di_g)$ we may assume that
\begin{equation}\label{reduction}
p(e_i)=e_i,\quad\forall\, i=1,\dots,n.
\end{equation}
Let $p$ denote also its extension to $\R^n$ as a 1-homogeneous function. For a $w\in\R^n$ we set
$$
f_w(v)=|\langle p(v),w\rangle|,\quad v\in\R^n.
$$

\noindent{\bf Claim 1.} {\em $f_w$ is a support function for every $w\in\R^n$.} 

Let $E\in\Gr(n,2)$. Fixing an orthonormal coordinate system on $E$ we identify it with $\R^2$ and
$\K(E)$ with $\K^2$. Let $w\in\sfe$ and consider the application 
$$\func{\rvd_w}{\K^2}{\R}{K}{h(\di K,w).}
$$
Lemma~\ref{osservazione 1} yields that $\rvd_w\in\Val_1$ and is monotonic.
Hence, by Theorem~\ref{Firey_mon}, there exists $L_w\in\K^2$ such that 
$$h(\di_w(K),w)=V(K,L_w)=\frac12 \int_{\sfed}h(L_w,v)dS_1(K,v),$$
for all $K\in\K^2$.
Notice that if we are in Case 2 of Theorem~\ref{Firey_mon}, then $L_w$ corresponds to the segment $S_{2}$ -which depends on $w$- in the notation of Theorem~\ref{Firey_mon}. 

Let $v_0\in \sfed\subset E\cong\R^2$ and $K=S_{\vo}=[-\vo,\vo]$. By the identification of $E$ with $\R^2$, the orthogonal vectors to $v_0$ are also in $E$, i.e., $\vo\in E$. 
Using \eqref{area_segment} and the definition of the difference body, we obtain
$$
h(\di_w(S_{\vo}),w)=\frac12 \int_{\sfed}h(L_w,v)dS_1(S_{\vo},v)=\frac12 (2h(L_w,v_0)+2h(L_w,-v_0))=h(D(L_w),v_0).
$$
On the other hand, by the definition of the function $p$,
$$
h(\di_w(S_{\vo}),w)=|\langle p(\vo),w\rangle|.
$$
We deduce that
\begin{equation}\label{star}
h(DL_w,v)=|\langle p(\v),w\rangle|
\end{equation}
for all $v\in\sfe\cap E$.
By \eqref{supp func g^t} and \eqref{star} with $\overline v$ instead of $v$, there exists a rotation $\tilde g$ on $E$ such that 
$$h(\tilde g(DL_w),v)=|\langle p(v),w\rangle|$$ 
for all $v\in\sfe\cap E$.
The 1-homogeneity of both, the support function and $p$, ensures that the equality continues to hold for every $v\in E$. This proves that $f_w$ restricted to $E$ is convex. As $E$ was arbitrary, $f_w$ is convex in $\R^n$.

\bigskip

\noindent{\bf Claim 2.} {\em There exists $g_0\in\O(n)$ such that for every $v\in\R^n$,}
$$
g_0(\di([-v,v]))=[-v,v].
$$
Let $i\in\{1,\dots,n\}$ be fixed. We apply Claim 1 to $w=e_i$, getting that
$$
f_i(v)=|\langle p(v),e_i\rangle|
$$
is a support function for every $1\leq i\leq n$. Let $L_i\in\K^n$ be such that $h(L_i,v)=f_i(v)$ for every $v\in\R^n$.
As $p$ is even (by definition; see Proposition~\ref{p}), $L_i$ is $o$-symmetric. 

On the other hand, by \eqref{reduction},
$$
f_i(\pm e_j)=|\langle p(e_j),\pm e_i\rangle|=|\langle e_j,e_i\rangle|=0, \quad j\neq i.
$$
This implies that $L_i$ is a segment centered at the origin and parallel to $e_i$. Thus, there exists a constant 
$c>0$ such that
$$
h(L_i,v)=|\langle p(v),e_i\rangle |=c|\langle v,e_i\rangle|,
$$
for all $v\in\R^n$.
Choosing $v=e_i$ we get immediately $c=1$. Hence we have proved that for every $i=1,\dots,n$ and for every $v\in\R^n$:
\begin{equation}\label{brown star}
|\langle p(v),e_i\rangle |=|\langle v,e_i\rangle |.
\end{equation}
This implies that for every $i\in\{1,\dots,n\}$ there exists a function $\epsilon_i\,:\,\sfe\longrightarrow\{-1,+1\}$ such that
$$
\langle p(v),e_i\rangle =\epsilon_i(v)\,\langle v,e_i\rangle,
$$
for all $v\in\sfe$.
Let
$$
 O^n_+=\{v\in\sfe\,:\,\langle v,e_i\rangle>0,\, i=1,\dots,n\}.
$$
For $v\in O^n_+$ taking into account the definition of $p$, the choice $e=e_n$, and \eqref{brown star} we have $\langle p(v),e_n\rangle>0$. 
Proposition~\ref{p} yields that $p$ is continuous in $ O^n_+$.
This implies that $\epsilon_i$ is constant in $ O^n_+$, for every $i$. Therefore,
for every $i=1,\dots,n$ there exists $\epsilon_i\in\{-1,+1\}$, (now not depending on $v$) such that
$$
\langle p(v),e_i\rangle=\epsilon_i\,\langle v,e_i\rangle,\quad\forall\, v\in O^n_+.
$$
Let $g_0\in\O(n)$ be defined by
$$
g_0(y)=g_0((y_1,\dots,y_n))=(\epsilon_1 y_1,\dots,\epsilon_n y_n),
$$
for all $y\in\R^n$, and consider a new operator $\di_{g_0}\in\MAdd^{s,+}$ defined by
$$
\di_{g_0}(K)=g_0(\di K),
$$
for any $K\in\K^n$. From Lemma~\ref{lemma linear}, $\di_{g_0}$ is monotonic and satisfies (VC). We have, for every $v\in O^n_+$,
$$
\di_{g_0}([-v,v])=g_0(\di([-v,v]))
=[-g_0(p(v)),g_0(p(v))].
$$
On the other hand
$$
\langle g_0(p(v)),e_i\rangle=\epsilon^2_i\langle v,e_i\rangle=\langle v,e_i\rangle,\quad\forall\,v\in O^n_+,\;\quad\forall\, i=1,\dots,n.
$$
Hence
\begin{equation}\label{green star}
\di_{g_0}([-v,v])=[-v,v],\quad\forall\, v\in O^n_+.
\end{equation}
Let $p_{g_0}$ be associated to $\di_{g_0}$ as indicated in Proposition~\ref{p}, with $e=e_n$. From \eqref{green star} it is clear that $p_{g_0}$ is the identity in $ O^n_+$.
Moreover, as the matrix representation of $g_0$ is of the form $g_0=\mathrm{diag}(\epsilon_1,\dots,\epsilon_n)\in\O(n)$, with $\epsilon_i\in\{1,-1\}$ for $1\leq i\leq n$, \eqref{brown star} holds true also for $p_{g_0}$. Let 
$$
O^n_{n,+}=\{v\in\sfe\,:\,\langle v,e_n\rangle>0\}.
$$
Arguing as above, we have that $p_{g_0}$ is continuous in $O^n_{n,+}$. For $i=1,\dots,n-1$, let
$$
A_i^+=\{v\in O^n_{n,+}\,:\, \langle v,e_i\rangle>0\},\quad
A_i^-=\{v\in O^n_{n,+}\,:\, \langle v,e_i\rangle<0\}.
$$ 
As $p_{g_0}$ is continuous in these sets, and they are connected, $\epsilon_i$ is constant in each of them. By $A_i^+\cap O^n_+\ne\emptyset$,
we have $\epsilon_i=1$ in $A_i^+$. 

We show next $\epsilon_i=1$ in $A_i^-$, for every $i=1,\dots,n-1$. Assume by contradiction that  there exists $i_0$ such that $\epsilon_{i_0}=-1$ in
$A^-_{i_0}$, and choose
$$
v=(v_1,\dots,v_n)\in\left(\bigcap_{i\ne i_0} A_i^+\right)\cap A_{i_0}^-.
$$ 
We consider now $v'\in\R^n$ such that the coordinates of $v'$ coincide with those of $v$ for $1\leq i\neq i_0 \leq n$ and the $i_0$-th coordinate is $-v_{i_0}$ in place of $v_{i_0}$. The assumption $\epsilon_{i_0}=1$ yields $p(v)=p(v')$. This contradicts the 
injectivity of $p$. We conclude that $\epsilon_i=1$ on $O^n_{n,+}$ for every $i$, i.e., $p$ is the identity on $O^n_{n,+}$. 
Hence, we have that there exists $g_0\in\O(n)$ such that $g_0(\di[-v,v])=[-v,v]$ for every $v\in\sfe$ with $\langle v,e_n\rangle\neq 0$. By the continuity of $\di$, the statement holds for every $v\in\sfe$ and
the claim is proved.

\bigskip

\noindent{\bf Claim 3.} {\em There exists $g_0\in\O(n)$ such that}
$$
g_0(\di K)=\frac12\, DK,\quad\forall\, K\in\K^n.
$$

Let $g_0$ be the element of $\O(n)$ from Claim 2. We will write, for convenience, $\di$ instead of $\di_{g_0}$. We know that
\begin{equation}\label{black star}
\di([-v,v])=[-v,v],\quad\forall\, v\in\R^n.
\end{equation}
Now let $\overline\di\,:\,\K^n\longrightarrow\K^n$ be defined by
$$
\overline\di K=\frac12\, DK,\quad\forall\, K\in\K^n.
$$
This operator has the same properties as $\di$: $\overline\di\in\MAdd^{s,+}$ is monotonic and satisfies (VC).

Let $w\in\R^n$ and consider the applications $\di_w,\,\overline\di_w\,:\,\K^n\longrightarrow\R$ defined by
$$
\di_w(K)=h(\di K,w),\quad
\overline\di_w(K)=h(\overline\di K,w),
$$
for all $K\in\K^n$. 
Lemma~\ref{osservazione 1} ensures that $\di_w,\,\overline\di_w\in\Val_1$ and they are also even. Using \eqref{black star} and the definition of $\overline\di_w$ it is easy to check that they have the same Klain map, 
which implies $\di_w(K)=\overline\di_w(K)$ for every $w\in\R^n$ and $K\in\K^n$. Hence, $\di\equiv\overline\di$, and the proof is completed.
\end{proof}

\medskip
In order to prove Theorem~\ref{+mon intro}, we first remark that condition (VC) is preserved when composing 
two elements in $\MAdd$ that satisfy the (VC) condition.

\begin{lemma}\label{composition_D}
If $\di,\dib\in\MAdd$ satisfy (VC), then $\di\circ\dib\in\MAdd$ and satisfies (VC).
\end{lemma}
\begin{proof}The continuity, translation invariance, and Minkowski additivity of the operator $\di\circ\dib$ are clear from the corresponding 
properties of $\di$ and $\dib$.

The (VC) condition follows from the (VC) condition of $\di$ and $\dib$. Let $c_\di, C_\di, c_\dib, C_\dib$ be the constants such that
(VC) is satisfied for $\di$ and $\dib$, respectively. Then,
$$
c_\di c_\dib\vol(K)\leq c_\di\vol(\dib K)\leq\vol(\di(\dib(K)))\leq C_{\di}\vol(\dib(K))\leq C_{\di}C_{\dib}\vol(K).$$
\end{proof}

\bigskip
Next, we proceed to prove Theorem~\ref{+mon intro}.

\begin{proof}[Proof of Theorem~\ref{+mon intro}]
Let $\di\in\MAdd$ be monotonic and satisfy the (VC) condition. By Lemma~\ref{composition_D}, the operator 
$$
\func{\dib}{\K^n}{\K^n}{K}{D(\di(D K))}
$$
satisfies (VC) and, moreover, $\dib\in\MAdd$ is clearly monotonic, even, and an $o$-symmetrization.

Thus, we can use Theorem~\ref{th monotone weak} to assert that there exists $g\in\GL(n)$ such that 
\begin{equation}\label{ex_g}
\dib(K)=g(DK),\quad\forall K\in\K^n.
\end{equation}
Since $\di$ is monotonic, Theorem~\ref{Firey_mon} yields that for every $u\in\sfe$ there exist $k\in\{1,\dots,n\}$ and $(n-k)$ 
pairwise orthogonal unit segments $S_{k+1}^u,\dots,S_n^u$ such that, if $k\geq 2$, then  
\begin{equation}\label{mon_final}
h(\di K,u)=V(K,L_u[k-1],S_{k+1}^{u},\dots,S_n^u),
\end{equation}
with $L_u\in\K^k(\spa\{S_{k+1}^u,\dots,S_n^u\}^{\perp})$ and $\dim L_u=k$; in particular, $\dim(L_u+ S_{k+1}^{u}+\dots+S_n^u)=n$. 
If $k=1$, then $h(\di K,u)=cV(K,S_{2}^{u},\dots,S_n^u)$ for some $c>0$. 

Hence, for $k\geq $1,  \eqref{mon_final} and \eqref{ex_g} yield
\begin{align*}
h(\dib(K),u)&=h(D(\di(DK)),u)=h(\di(DK)-\di(DK),u)\\
&=h(\di(DK),u)+h(\di(DK),-u)\\
&=V(DK,L_u[k-1],S_{k+1}^{u},\dots,S_n^u)
+V(DK,L_{-u}[k-1],S_{k+1}^{-u},\dots,S_n^{-u})
\\&=h(g(DK),u),
\end{align*}
for every $K\in\K^n$.

Next, we show that $k=1$. Assume that $k\geq 2$. Then Theorem~\ref{mix_volumes_Schneider} together with $\dim(L_u+ S_{k+1}^{u}+\dots+S_n^u)=n$ yields
$$V(DK,L_{\pm u}[k-1],S_{k+1}^{\pm u},\dots,S_n^{\pm u})\neq 0$$ for every convex body $K$ such that $\dim(K+L_{\pm u})\geq k$. 

Let $K=S_v$ be the line segment $[-v,v]$ with $v\in \spa \{ g^{-1}(u)\}^{\perp}$. 
From $h(g(D S_v),u)=h(D S_v, g^{-1} u)=0$, we obtain that 
$V(DK,L_{\pm u}[k-1],S_{k+1}^{\pm u},\dots,S_n^{\pm u})=0$. Hence, 
$$S_v\subset\spa\{S_{k+1}^{u}+\dots+S_n^{u}\}\quad\textrm{and} \quad
S_v\subset\spa\{S_{k+1}^{-u}+\dots+S_n^{-u}\}$$ 
for every $v\in(g^{-1}(u))^{\perp}$. Since $\dim (\spa\{g^{-1}(u)\})^{\perp}=n-1$, it necessarily follows that $\dim\{S_{k+1}^{u}+\dots+S_{n}^u\}=\dim\{S_{k+1}^{-u}+\dots+S_{n}^{-u}\}=n-1$, which is not possible for $k\geq 2$. Hence, $k=1$ as claimed and there exist pairwise 
orthonormal centered unit segments $S_2^u,\dots,S_n^u$ and $c>0$ such that
\begin{align*}
h(\dib(K),u)
&=cV(DK,S_{2}^{u},\dots,S_n^u)
+cV(DK,S_{2}^{-u},\dots,S_n^{-u})\\&=
2c(V(K,S_{2}^{u},\dots,S_n^u)+V(K,S_{2}^{-u},\dots,S_n^{-u})).
\end{align*}
Moreover, since $\spa\{S_2^u+\dots+S_n^u\}=\spa\{S_2^{-u}+\dots+S_n^{-u}\}$, we can assume $S_j^u=S_j^{-u}$ for every $2\leq j\leq n$ and $u\in\R^n$. 
Hence, there exists a constant $\tilde c>0$ such that
\begin{equation}\label{supp_dpd}
h(\dib K,u)=\tilde cV(K,S_2^u,\dots,S_n^u),\quad\forall K\in\K^n.
\end{equation}
On the other hand, using Corollary~\ref{th repres D_v}, 
\begin{equation}\label{supp_D}
h(DK,w)=h(K,w)+h(K,-w)=\frac{1}{2^{n-1}}V(K,S_{w_2},\dots,S_{w_n}),\quad\forall w\in\sfe, K\in\K^n
\end{equation}
where $S_{w_j}=[-w_j,w_j]$, $j\in\{2,\dots,n\}$, and $\{w_2,\dots,w_n\}$ is an orthonormal basis of $\spa\{w^{\perp}\}$. 

Comparing \eqref{supp_dpd} and \eqref{supp_D}, we choose  $w=g^{-1}(u)$ and $S_j^u=S_{w_j}$. 
The statement of the theorem follows taking equation \eqref{mon_final} into account. 
\end{proof}

\section{The Minkowski endomorphism case}\label{sec: SOn}

In this section we prove Theorem~\ref{+On_dim_geq_3} and its 2-dimensional version: 

\begin{theorem}\label{+On_dim_2}
Let $n=2$. 
\begin{enumerate}
\item[\emph{(i)}] An operator $\di\in\MEnd$ satisfies (VC) if and only if there are $g\in\SO(2)$ and $a,b\geq 0$ with $a+b>0$ such that 
$$
\di K=ag(K-\st(K))+bg(-K+\st(K)),\quad\forall K\in\K^n.
$$
\item[\emph{(ii)}]  An operator $\di\in\MEnd^{s}$ satisfies (VC) if and only if there are $\lambda >0$ and $g\in\O(2)$ such that $\di K=\lambda gDK$ for every $K\in\K^n$.
\end{enumerate}
\end{theorem}

Theorems~\ref{+On_dim_geq_3} and~\ref{+On_dim_2} can be directly obtained from Theorem~\ref{SO_schneider}.
Indeed, by Remark~\ref{segment}, if $\di\in\MEnd$ satisfies (VC), then $\di$ maps every (non-degenerated) segment to a non-degenerated segment. Hence, we can apply Theorem~\ref{SO_schneider}. The second statement of Theorems~\ref{+On_dim_geq_3} and~\ref{+On_dim_2} follows directly from the first one and the assumption of $o$-symmetrization.

We provide an alternative proof for Theorem~\ref{+On_dim_geq_3}(ii) and its 2-dimensional version in Theorem~\ref{+On_dim_2} by using the representation results we have proven in the previous sections.

\begin{proof}[Proof of Theorem~\ref{+On_dim_geq_3}(ii)]
Let $n\geq 3$ and let $\di\in\MEnd$ satisfy (VC).

First we observe that every symmetrization $\di\in\MEnd^{s}$, which is $\SO(n)$-equivariant, and by Theorem~\ref{SOimpliesO} also $\O(n)$-equivariant, is also even, 
that is, we can assume $\di\in\MEnd^{s,+}$. 

Theorem~\ref{thm n ge 2} implies the existence of $\rho_1$ and $\pi_1$ such that for every generalized zonoid $Z$ with (signed even) generating measure $\rho_Z$ and every $u\in\R^n$,  
$$h(\di(gZ),u)=\frac{1}{2}\int_{\sfe}\rho_1(E_{gv})V_1([-u,u]|\pi_1(E_{gv}))d\rho_Z(v),$$
which together with the $\O(n)$-equivariance provided by Theorem~\ref{SOimpliesO}  yields
$$
h(\di(g(Z)),u)=h(\di Z,g^T(u))=\frac{1}{2}\int_{\sfe}\rho_1(E_v)V_1([-g^T(u),g^T(u)]|\pi_1(E_v))d\rho_Z(v).
$$
We recall that $E_v=\spa\{v\}$.

Particularizing the above expressions for a centered segment $Z=[-v,v]=:S_v$, $v\in\sfe$, which obviously is a generalized zonotope with generating measure given by $\rho_{S_v}=\frac{1}{2}(\delta_{v}+\delta_{-v})$  (see \eqref{gen_zonoid}), we obtain that
the equality
$$h(\di(g(S_v)),u)=h(\di S_v,g^T(u))$$
holds for every $u\in\sfe$ and $g\in\O(n)$ if and only if 
\begin{equation}\label{On_1}
\rho_1(E_{g(v)})V_1([-u,u]|\pi_1(E_{g(v)}))=\rho_1(E_v)V_1([-u,u]|g(\pi_1(E_{v})))
\end{equation}
for all $u\in\sfe$ and $g\in\O(n)$.

For $u\in (\pi_1(E_{gv}))^{\perp}$, the left-hand side of \eqref{On_1} vanishes and as $\rho_1>0$, 
$$
V_1([-u,u]|g(\pi_1(E_{v})))=0, \quad \forall u\in\spa(\pi_1(E_{g(v)}))^{\perp}.
$$
This implies $(\pi_1(E_{g(v)}))^{\perp}=(g(\pi_1(E_v)))^{\perp}$, that is, 
$$
\pi_1(E_{g(v)})=g(\pi_1(E_v)),
$$
for all $v\in\sfe$. 

Since $E_{g(v)}=g(E_v)$, we obtain that $\pi_1(g(E_v))=g(\pi_1(E_v))$ for all $v\in\sfe$. 
This is equivalent to  the fact that $\pi_1:\Gr(n,1)\longrightarrow\Gr(n,1)$ is $\O(n)$-equivariant. 
This is only possible if $\pi_1(E_v)=E_v$ for every $v\in\sfe$. 
Indeed,  let $g\in \O(n)$ be such that $g(e_1)=e_1$. Then 
$$\pi_1(E_{ge_1})=\pi_1(E_{e_1})$$
and, by the $\O(n)$-equivariance,  
$$\pi_1(E_{ge_1})=g\pi_1(E_{e_1}).$$ 
Combining both equalities, we obtain that for every $g\in\O(n)$ with $ge_1=e_1$, $g$ also fixes the direction $\pi_1(E_{e_1})$, which is only possible if $\pi(E_{e_1})=E_{e_1}$. 
Again by the $\O(n)$-equivariance of $\pi_1$, we obtain 
$$\pi_1(E_{v})=\pi_1(E_{he_1})=h\pi_1(E_{e_1})=he_1=v,$$
where $h\in\O(n)$ satisfies $v=he_1$, and the claim follows. 
We note that if $n=2$, then, only $\pm\mathrm{Id}$ are in $\O(2)$ and satisfy $gE_{e_1}=E_{e_1}$. Hence, for $n=2$, we do not obtain any restriction on $\pi_1$ from the $\O(2)$-equivariance.

Plugging now $\pi_1(E_v)=E_v$ in \eqref{On_1}, we get $\rho_1(E_{g(v)})=\rho_1(E_v)$ for every $g\in\O(n)$ and $v\in\sfe$, 
which implies that $\rho_1$ is constant on $\Gr(n,1)$. Hence, for every generalized zonoid $Z$, 
\begin{equation}\label{On_2}
h(\di Z,u)=c\int_{\sfe}V_1([-u,u]|E_v)d\rho_Z(v),\quad\forall u\in\R^n,
\end{equation}
for some $c>0$.

On the other hand, using the expression \eqref{supp_D} for the support function of the difference body together with \eqref{eq_Gardner}, we have that the Klain function of the support function of the difference body in direction $u\in\sfe$ is given by
$E_v\mapsto \widetilde cV_1([-u,u]|E_v)$, $v\in\sfe$, $\widetilde c>0$.

 Hence, from Theorem~\ref{6.4.12} we have, for every generalized zonoid $Z$, 
\begin{equation}\label{On_3}
h(DZ,u)=\int_{\sfe}V_1([-u,u]|E_v)d\rho_Z(v).
\end{equation}
Thus, from \eqref{On_2} and \eqref{On_3}, we deduce that there exists $c>0$ such that $\di Z=cDZ$ for every generalized zonoid $Z\in\K^n$. 
Since generalized zonoids are dense in the space of $o$-symmetric convex bodies, $\di K=cDK$ for every $K\in\K^n_s$. The evenness 
assumption on $\di$ yields $\di K=cDK$ for every $K\in\K^n$. Indeed, since $\di$ is even, Minkowski additive and $\di K=cDK$ for every $K\in\K^n_s$, we have 
$$
\di K=\frac{1}{2}(\di(K)+\di(-K))=\frac{1}{2}\di(DK)=\tilde cDK,
$$
and the ``only if'' part of the statement follows. The ``if'' part is direct.

The case $n=2$, namely Theorem~\ref{+On_dim_2}, follows in a similar way, using Theorem~\ref{thm n=2} instead of Theorem~\ref{thm n ge 2}.
\end{proof}

\section{Appendix}

We first prove the topological properties of the Grassmannian that we needed. 
\begin{proposition}
The space $\Gr(n,k)$ endowed with the topology given by the distance~\eqref{d_Gr} is compact and connected.
\end{proposition}
\begin{proof}
We first prove the compactness. Let $(E_i)_{i\in\N}$ be a sequence in $\Gr(n,k)$. We show that it has a subsequence convergent to an element in $\Gr(n,k)$, as we are in a metric space. For each $(E_i)_{i\in\N}$ consider a vector $v_1^i\in\sfe\cap E_i$. The sequence $(v_1^i)_{i\in\N}$, as it is a sequence on the compact space $\sfe,$ has a convergent subsequence. Let $I_1$ be the set of indices of this convergent subsequence, and let $v_1$ denote the limit vector. Next, for $i\in I_1$, consider an arbitrary vector $v_2^i\in E_i\cap \spa\{v_1^i\}^{\perp}\cap\sfe$. Notice that, for every $i\in\N$, $v_2^i\in \spa\{v_1^i\}^{\perp}$. As before, $\{v_2^i\,:\,i\in I_1\}\subseteq \sfe$, and thus has a convergent subsequence, with indices in $I_2\subset I_1$. The limit vector $v_2$ is orthogonal to $v_1$ since $\langle v_2^i,v_1^i\rangle=0$ for every $i\in I_1$. Repeating this process, we obtain $k$ pairwise orthogonal unit vectors $v_1,\dots,v_k$, limits of sequences in $\sfe$ with in
 dices in a subset $I_k\subset I_{k-1}\subset\dots\subset I_1\subset\N$. 
Let $E:=\spa\{v_1,\dots,v_k\}\in\Gr(n,k)$. We claim that $E$ is the limit of the subsequence 
$(E_i)_{i\in I_k}$ of $(E_i)_{i\in\N}$. By construction, each $(E_i)_{i\in I_k}$ can be written as $E_i=\spa\{v_1^i,\dots,v_k^i\}$, with $v_j^i$, $1\leq j\leq k$, constructively defined above and pairwise orthogonal. The claim will follow if for every $\epsilon>0$, there exists $i_0\in I_k$ such that 
$d(E_i,E)<\epsilon$ for every $i>i_0$, $i\in I_k$. Using expression~\eqref{dist_angles_Grnk} and~\eqref{dist_Grn1}, we obtain easily that this condition is satisfied. Indeed, since the sequence $(v_j^i)_{i\in I_k}$ tends to $v_j$ for every $1\leq j\leq k$, the distance, i.e., the angle, between $\spa\{v_j^i\}$ and $\spa\{v_j\}$ also tends to zero. Since in~\eqref{dist_angles_Grnk} we are considering a minimum among all distances between 1-dimensional vector spaces, i.e., lines  where one of them is fixed, the claim follows. 

Next we prove that $\Gr(n,k)$ is connected by proving that it is pathwise-connected. Let $E_0:=\spa\{e_1,\dots,e_k\}\in\Gr(n,k)$ be a fixed element and let $E\in\Gr(n,k)$ be arbitrary. Choose pairwise orthogonal unit vectors $\{v_1,\dots,v_k\}$ such that $E=\spa\{v_1,\dots,v_k\}$. There exists a rotation $g\in\SO(n)$ such that $ge_j=v_j$, $1\leq j\leq k$. This rotation is not unique, but we fix one. Now, the space $\O(n)$ is pathwise-connected. Hence, there is a continuous map $[0,1]\longrightarrow\SO(n)$, $t\mapsto g(t)$, such that $g(0)=\mathrm{id}$ and $g(1)=g$. The map $[0,1]\longrightarrow\Gr(n,k)$, $t\mapsto\spa\{g(t)e_1,\dots,g(t)e_k\}$ defines a continuous path  between $E_0$ and $E$. The continuity can be proved, similarly as before, by using~\eqref{dist_angles_Grnk}, since $t\mapsto g(t)e_j$ is a continuous map on $\sfe$ for every $1\leq j\leq k$. 
\end{proof}

Finally, we give a proof of the following expression of the support function of the difference body by using Theorem~\ref{Firey_mon}. 
\begin{corollary}\label{th repres D_v} Let $n\geq 2$, $v\in\R^n$ and let $K\in\K^n$. 
Then there exist $n-1$ vectors $\{v_2^v,\dots,v_n^v\}$, constituting an orthonormal basis of  $\spa\{v\}^{\perp}$,  such that 
$$h(DK,v)=h(K,v)+h(K,-v)=\frac{1}{2^{n-1}}V(K,S_{v_2}^v,\dots,S_{v_n}^v),\quad\forall v\in\sfe,$$
for $S_{v_j}^v=[-v_j^v,v_j^v]$, $j\in\{2,\dots,n\}$. 
\end{corollary}

\begin{proof}
Let $v\in\R^n$. Since the difference body operator is a monotonic, translation invariant, and 1-homogeneous valuation, the function $K\mapsto h(DK,v)$ is a
real-valued monotonic, translation invariant, and 1-homogeneous valuation. Thus, we can apply Theorem~\ref{Firey_mon} to obtain the statement. Indeed, it is enough to prove that (i) in Theorem~\ref{Firey_mon} is not possible. 
Assume by contradiction, that $k\geq 2$. 
Let $S^v_{v_1},\dots, S^v_{v_{k+1}}$ be pairwise orthogonal unit vectors and $L^v\subseteq \mathrm{span}\{S^v_{k+1},\dots,S^v_n\}^\perp$ be a convex body 
of dimension $k$, such that
\begin{equation}\label{kgeq2}
\di_v(K)=V(K,L^v[k-1],S^v_{k+1},\dots,S^v_n),
\end{equation}
ensured by Theorem~\ref{Firey_mon} (ii).
The subscript $v$ denotes the dependence of the convex bodies on the fixed direction $v\in\R^n$.
Let $w\in\sfe\cap\spa\{v\}^{\perp}$ and $K=S_w=[-w,w]$. By the definition of the difference body and \eqref{kgeq2}, 
\begin{align}\label{eqto0}
h(D S_w,v)&=0\nonumber
\\&=V(S_w,L^{v}[k-1],S^v_{k+1},\dots,S^v_n).
\end{align}
By the assumptions on $L^v$ and the line segments $S_{v_i}$, $1\leq i\leq k$, it is clear that $$\dim(L^v+S^v_{k+1}+\dots+S^v_n)=k+n-k=n.$$ Hence, by Theorem~
\ref{mix_volumes_Schneider}, 
$V(K,L^{v}[k-1],S^v_{k+1},\dots,S^v_n)\neq 0$ for every convex body $K\in\K^n$ with $\dim K\geq 1$. But this is a contradiction with \eqref{eqto0}. Hence, $k=1$ and, for every $v
\in\sfe$ there exist pairwise orthogonal segments $S_2^v,\dots,S_n^v$ such that $\spa\{S_2^v+\dots+S_n^v\}=\spa\{v\}^{\perp}$ and $h(DK,v)=c^vV(K,S_2^v,\dots,S_n^v).$

It remains to compute the constant $c^v$ such that
$h(DK,v)=c^vV(K,S_2^v,\dots,S_n^v).$ For that, we take $S_v=[-v,v]$, $v\in\sfe$, and directly verify that
$$h(DS_v,v)=h(2S_v,v)=2|\langle v,v\rangle|=2.$$
By the multilinearity of mixed volumes and Theorem~\ref{mix_volumes_Schneider},
$$V(S_v,S_2^v,\dots,S_n^v)=V_n(S_v+S_2^v+\dots+S_n^v)=2^n.$$
Hence, $c^v=2^{1-n}$ for every $v\in\sfe$ and the statement holds. 
\end{proof}

\def\cprime{$'$}

\end{document}